\documentclass[11pt,english]{article}
\usepackage[T1]{fontenc}
\usepackage{amsmath}

\DeclareMathOperator*{\argmin}{arg\,min}
\usepackage{amsthm}
\usepackage{amssymb}
\usepackage{mathtools}
\usepackage{leftindex}
\usepackage{multirow}
\usepackage{appendix}
\usepackage{breqn}
\usepackage{thmtools,thm-restate}
\usepackage{hyperref}
\usepackage{algorithm,algorithmic}

\usepackage[nobysame]{amsrefs}

\makeatletter
\theoremstyle{plain}
\newtheorem{theorem}{\protect\theoremname}[section]
\theoremstyle{plain}
\newtheorem{lemma}[theorem]{\protect\lemmaname}
\theoremstyle{plain}
\newtheorem{proposition}{\protect\propositionname}[section]
\theoremstyle{plain}

\theoremstyle{plain}

\theoremstyle{definition}
\newtheorem{definition}{\protect\definitionname}[section]
\theoremstyle{definition}

\newtheorem{assumption}{Assumption}[section]
\newtheorem{remark}{Remark}[section]
\ifx\proof\undefined
\newenvironment{proof}[1][\protect\proofname]{\par
	\normalfont\topsep6\p@\@plus6\p@\relax
	\trivlist
	\itemindent\parindent
	\item[\hskip\labelsep\scshape #1]\ignorespaces
}{%
	\endtrivlist\@endpefalse
}
\providecommand{\proofname}{Proof}
\fi

\usepackage{fullpage}

\usepackage{times}

\makeatother

\usepackage{babel}
\providecommand{\definitionname}{Definition}
\providecommand{\theoremname}{Theorem}
\providecommand{\lemmaname}{Lemma}
\providecommand{\propositionname}{Proposition}
\providecommand{\propertyname}{Property}
\providecommand{\examplename}{Example}
\providecommand{\corname}{Corollary}

\usepackage{pgfgantt}
\usepackage{bigints}
\usepackage{listings}
\usepackage{verbatim}
\usepackage{dsfont}
\usepackage{float}
\usepackage{subcaption}
\usepackage{fullpage}
\usepackage{authblk}
\numberwithin{equation}{section}
\newcommand{\N}{\ensuremath{\mathbb{N}}}

\definecolor{mygreen}{RGB}{28,172,0}
\definecolor{mylilas}{RGB}{170,55,241}
\usepackage{color}

\begin{document}
\title{The Value of Temporary Control for the M/M/1 Queue}
%
%

\author[1]{Odysseas Kanavetas}
\author[1]{Camiel M.P. Koopmans}
 \author[1]{Floske M. Spieksma}
 \affil[1]{ Mathematisch Instituut, Universiteit Leiden, the Netherlands}
 

\pagenumbering{gobble}
\maketitle              
\begin{abstract}
    In this article, a one-off option for temporary service rate control for the M/M/1 queue is considered.
     After taking this option, during a single exponentially distributed period, two service rates are available for use. Once service rate control is lost, the system continues with a fixed service rate $\mu$.  
     The objective is to minimise the sum of holding costs and service costs. 
    We approximate the expected total saved cost by taking the one-off option, depending on the starting state or starting distribution.
    Using the Value Iteration algorithm with $M$-uniform geometric recurrence, we present methods to approximate the expected total saved future cost, as well as the expected total saved discounted future cost. Furthermore, we obtain theoretical results on the structure of optimal policies and strong Blackwell optimality. The paper is concluded by numerically applying the methods to various instances of the model.
\end{abstract}

\newpage
\pagenumbering{arabic}
\section{Introduction}\label{Introduction}

Queueing systems have been well studied due to their practical relevance for (tele-)communication, traffic flow, computer sciences, and public services. The performance of these systems, especially for large scale systems, can be improved significantly through dynamic control. Service rate control is one of the implementations of dynamic control.

Service rate control serves as a method to fit the method/speed of service to the congestion level of a service system. The speed of service depends on a decision of staffing allocation or on a choice between service methods. The speed of service is directly affected, while the choice of service method is expected to have economic effects. These economic effects can result from direct higher costs or indirectly, for example, from varying service quality (see for example~\cites{gryna1999quality,rust1995return}).  

Furthermore, a cost for the amount of customers in the system is incurred through \textit{holding costs}. These costs can represent waiting cost of customers, customer dissatisfaction or the use of valuable resources to customers in the system like space and attention.

Service rate control of queueing systems has been studied in e.g. \cites{Lee,Adus,Ata,Badian,Kumar,Kitaev,Teghem}. 
We believe that one-off opportunities may arise in practice. For example, a rare opportunity arises to hire an intern, introducing a single period with service rate control through staffing allocation. Another example can be given through a one-off offer for a machine with a finite lifespan. One-off opportunities can arise through scarceness or through unique discounts. Existing service control literature typically assumes control remains indefinitely available, whereas this paper considers transient and exhaustible control opportunities. Consequently, the objective shifts from average cost optimisation to quantifying future cost savings induced by temporary intervention.
Furthermore, in an upcoming paper we plan to study a queueing system with control vacations, where, as soon as control is lost, the start of the next period of service rate control is scheduled at a cost. The cost saved by one control period is immediately related to the question whether repetitive control periods can reduce the average expected cost over the infinite time horizon. 

The exponential distribution (together with the hyperexponential and Weibull distribution) is a standard choice to take as the lifetime distribution in studies regarding machine failures, see for example~\cite{Groenevelt, Kapur,Gnedenko}. For this reason, we consider the exponential distribution to be appropriate to model the length of a temporary period with service rate control.

To decide whether it is desirable to take such a one-off decision of service rate control, we need to study the implications for the future cost. In systems where costs are of a financial nature, future costs can be discounted to match inflation. As such, one would need to know the total expected discounted cost that would be saved by taking the one-off decision.
When costs are of a social nature, for example through waiting times, this discount does not apply. In this case, the effect of the one-off decision is measured through the total saved cost.

This paper studies the expected total discounted, and non-discounted, saved future cost for single service systems with a single queue by introducing temporary service rate control. In short, the contributions of this paper are
\begin{itemize}
\item The novel inclusion of temporary service rate control.
\item Insight into the structure of optimal policies and Blackwell optimality~\cite{Blackwell}.
    \item Derivation of $\epsilon$-approximations of the expected \textbf{total} saved future (non-)discounted cost.
    \item Derivation of $\epsilon$-optimal policies.
\end{itemize}
Here, \textit{$\epsilon$-approximation} refers to an approximation with an (additive) approximation error of at most $\epsilon>0$.
Likewise, an \textit{$\epsilon$-optimal policy} is a policy that achieves a cost of at most $\epsilon>0$ more than the total future\\ (non-)discounted cost of an optimal policy (given a starting distribution). 
The specific strengths of the derived methods lie in their analysis, resulting in the approximation guarantees. In both the discounted and non-discounted case, we use a contraction approach to approximate the expected total future saved cost starting in a specific state. These approximate saved costs are used to approximate the expected total future saved cost starting in the stationary distribution.
Much of the methodology used in this paper originates from, or is influenced by, some seminal papers.

Crabill~\cite{mm1service} and Lippman~\cite{lippman} consider control of a variable service rate in a single M/M/1 queue and treat the model as a Markov Decision Process (MDP). In both papers, the minimisation of the sum of the service and holding costs is considered, both for the expected total discounted cost and the average expected cost over the infinite time horizon. Crabill first proves that stationary optimal policies exist and that they have a switchover/threshold structure. 
The paper by Lippman proves the same thing in a very similar setting, but is mainly influential through its first use of uniformisation for MDPs to prove the existence of optimal policies with certain monotonicity properties.
The method of uniformisation was later further developed by Serfozo~\cite{serfozo}. 
After uniformisation, Lippman 
applies Value Iteration (VI).
The first use of VI for MDPs that is known to us, is by Shapley~\cite{Shapley1953StochasticG}. The use of VI by Shapley has not been recognised as such until much later~\cite{Kallenberg2002}.
Papers that consider the convergence of VI for MDPs are those by Altman et al.~\cite{Altman}, by Borkar~\cite{Borkar}, by Sennott~\cite{sennott} and by Spieksma~\cite{spieksma1990geometrically}.
Weber and Stidham show monotonicity results for more general queueing models~\cite{Stidham,WeberStidham}.
Koole~\cite{koole,KooleLong} develops the more general framework of Event Based Dynamic Programming (EBDP) to deduce structural results for queueing systems through establishing monotonicity in VI.

Later research that is also relevant to service rate control models for an M/M/1 queue, includes the papers by Ertiningsih et al.~\cite{Ertiningsih}, George and Harrison~\cite{georgeharrison} and a paper by Blok and Spieksma~\cite{mdppractice}, which serves as a roadmap to prove structural properties of optimal policies.

\subsection*{Problem description}

 We study a single service station with a single queue. Customers arrive at the station according to a Poisson process with rate $\lambda$ and service times are taken to be exponentially distributed.
 We model this system as an MDP.
In our research, we are interested in the expected  total cost reduction resulting from one period of control.
The expected total future cost is finite when the cost is discounted, but infinite when the cost is not discounted. Furthermore, exercising service rate control will not reduce the average expected cost of the system, as service rate control is only temporary. 
Our goal is to compare the expected total (discounted) future cost benefit from optimal temporary service rate control.
Then, we will apply VI to the model to $\epsilon$-approximate the expected saved cost. In the non-discounted case, VI does not give an approximation guarantee. Therefore, we construct a contraction map that leads to $\epsilon$-approximations. Finally, we consider the expected saved cost, given that the system was stationary before enabling service rate control. As the future saved cost can quickly increase as a function of the initial number of customers, we need analytical results for approximation guarantees for this case.

We will now formalise the MDP.
The state of the system is defined by the queue length and whether we (still) have service rate control.
The state space is
$S=\{ (i,q) :  i\in \N_0 ,\  q\in \{0,1\}\}$, with the first component, $i$, denoting the number of customers in the system, and the second component, $q$, denoting the situation with control ($q=1$) or without control ($q=0$). 

\noindent The single server can have a temporary period of service rate control, which is lost after a time that is exp$(\beta)$ distributed.
 During this period, the server can choose to serve at any of the rates $\mu_1,\mu_2 \in \mathbb{R}_{>0}$ with $\mu_1<\mu_2$. The use of service rates $\mu_1,\mu_2$ comes at a cost of $c_{\mu_1}$ and $c_{\mu_2}$ per time unit, respectively.
 W.l.o.g. we assume that $0=c_{\mu_1}<c_{\mu_2}$.  When the server does not have service rate control, all future services are exponentially distributed with a fixed rate $\mu \in \{\mu_1,\mu_2\}$ with cost rate $c_\mu$.
 To ensure that the system is stable without service rate control, we assume $\mu>\lambda$.
 In this framework, there are two situations to consider, namely $\mu_1>\lambda$ and $\mu_2>\lambda \geq \mu_1$. In this second situation an unstable service rate is admissable and the system will stabilise when control is lost.  Due to the certain loss of control, we find that $S$ consists of one transient class $\{(i,1) :  i\in \mathbb{N}_0\}$ and one closed positive recurrent class $\{(i,0) :  i\in \mathbb{N}_0\}$.
The action space is given by $$\mathcal{A}(i,q)=\begin{cases}
    \{\mu_1,\mu_2\} \hspace{2cm} \text{if } (i,q)=(i,1)\in S,\\
    \{\mu\} \hspace{2.74cm} \text{if } (i,q)=(i,0)\in S,
\end{cases}$$ where the actions specify the chosen service rate. Note that this means that, also in the state $(0,0)$, service rate $\mu$ is chosen with its service rate cost, even though the server is not in service. This choice is taken for consistency and applies to human servers and machines for which the cost rate is not dependent on whether the server is in service or not. A change in this choice would affect cost and can affect optimal policies, but would not have large effects on structures, methods and proofs that we give in this paper.
Under the assumptions used,
 we may restrict ourselves to stationary deterministic policies~\cite{mdppractice}. Thus, we consider the stationary deterministic policy space over $S$ given by $\Phi$, and we will denote stationary deterministic policies by $\phi  $. 

In addition to the cost of the service rate used, the operational cost of the system also consists of \textit{holding costs} for the number of customers in the system. These holding costs do not depend on the taken action or state of control. Thus, the holding costs can be defined by function $h:S\rightarrow\mathbb{R}$, such that the cost charged per time unit for having $i$ customers in the system is  given by $h(i,0)=h(i,1):=h(i)$. 
We state the following assumption on the holding cost function.
\begin{assumption}\label{holdingassump0}
W.l.o.g., $h(0)=0$. Furthermore,
the holding cost function $h$ is non-negative, non-decreasing and convex in the customer amount, and unbounded. 
\end{assumption}

The question we seek to answer is: ``\textit{What is the expected saved cost resulting from the one-off deal to get a service rate control period of exponentially distributed length?}''. 

The future cost is discounted with rate $0\leq \alpha<\infty$ (such that $\alpha=0$ corresponds with the non-discounted case).
Thus, we are interested in finding the total expected discounted cost difference between two continuous time MDPs, one without taking the one-off deal for service rate control and one with taking the one-off deal for service rate control.
Let $\{Y(t)\}_{t\in \mathbb{R}_{\geq0}}$ be the process starting with control (i.e.,
 in the set $\{(i,1) :  i\in \mathbb{N}_0\}$) using policy $\phi\in \Phi$ and let $\{X(t)\}_{t\in \mathbb{R}_{\geq0}} $ be the process starting with fixed service rate $\mu$ (i.e., starting in the set   $\{(i,0) :  i\in \mathbb{N}_0\}$). We note that $\mu=\phi(i,0)$ for any $(i,0)\in S$.
  The initial number of customers in the system for both processes is drawn from the same probability distribution $p$ on $\mathbb{N}_0$. This distribution $p$ induces starting distributions $p^0$, $p^1$ such that the non-zero entries of $p^0,p^1$ are given by $p^1(i,1)=
p(i)=p^0(i,0)$ for all $i\in \mathbb{N}_0$.
 The corresponding cost rates are $c(i,0)=c_\mu+h(i)$, and $c^{\phi}(i,q)=c_{\phi(i,q)}+h(i) $, 
for processes $\{X(t)\}_{t\in \mathbb{R}_{\geq0}}$, $ \{Y(t)\}_{t\in \mathbb{R}_{\geq0}} $, respectively. Now, the expected total discounted cost difference between the two processes is given by

\begin{equation}\label{questionMathematicallyDiscounted}
    \mathbb{E}^\phi_{p^0 \times p^1} \Bigg[ \int_0^\infty e^{-\alpha t}\Big( c(X(t))- c(Y(t))\Big) dt\Bigg] .\end{equation}  

The starting distributions $p$ that we will consider are the Dirac measure $\delta_i, i\in \mathbb{N}_0,$ and the stationary distribution $\pi_\mu$. The expected total (discounted) saved cost we study in this paper is 
$$
\sup_{\phi \in \Phi}   \ \mathbb{E}^\phi_{p^0 \times p^1} \Bigg[ \int_0^\infty e^{-\alpha t}\Big( c(X(t))- c(Y(t))\Big) dt\Bigg].$$
We will derive $\epsilon$-approximations of all these values 
of interest.
Often, it is possible to explicitly determine an optimal policy.
For the case this is not possible, we will also show that we can determine $\epsilon$-optimal policies.

This paper is organised as follows.
First, we give the discrete time equivalent model of our MDP using uniformisation in Section~\ref{sect:discrmodel}.
In Section~\ref{sect:VIpol}, 
we show the implementation of Value Iteration (VI) for this model. 
In Section~\ref{sect:ConvVI}, we will prove the convergence of VI for the discounted model. 

In Section~\ref{sect:ConvVInonDisc} we prove that VI also converges
for the non-discounted case using $M$-uniform geometric recurrence. This implies Blackwell optimality for the model, but it does not directly result in approximation guarantees on the expected saved cost.

Next, we use Event Based Dynamic Programming~\cite{KooleLong} in Section~\ref{sect:ebdp} to show that the policies in every step of VI are of a threshold form.  This allows us to conclude that all policies found by VI are of a threshold form. We proceed to show that this threshold is monotonic w.r.t. the number of steps of VI and w.r.t. the discount factor. We identify the discounted cases in which the threshold is infinite. For all other cases we show that we can bound the threshold. As a result, we can conclude Blackwell optimality for optimal policies that result from VI.

The convergence of VI  allows us to approximate the expected saved costs 
starting in a specific state, 
as shown in Section~\ref{sect:ConvVIDisc}. 

In Section~\ref{sect:nondmod}, we use an equivalent related model with discount rate $\beta$ in order to $\epsilon$-approximate the total saved cost in any specific starting state (in Appendix~\ref{sect:MProofs} we use $M$-uniform geometric recurrence resulting in an alternative approximation method).

We proceed to approximate the expected saved cost starting in the stationary distribution  (both for the discounted and non-discounted model) in Section~\ref{sect:SavedVals}. Generally, we can explicitly determine optimal policies. When this is not possible, we can always give $\epsilon-$optimal policies.  
Finally, we present all the approximations for various parameters through computational examples in Section~\ref{sect:num}.

\section{The discrete time model }\label{sect:discrmodel}

As our MDP has a countable state space and bounded transition rates, a standard approach for its analysis is to consider the uniformised discrete time equivalent~\cite{lippman,serfozo}. Optimal stationary deterministic policies and the expected cost of both the continuous time original model and the discrete time uniformised variant coincide \cite{serfozo}. To proceed, we normalise the parameters accordingly (scaling time), such that \\
$\lambda +\mu_1+\mu_2+\beta=1$. In order to ensure the same expected long-run cost, we scale 
$c_{\mu_2}$ and 
$h$ by the same value. Furthermore, the continuous time discount factor $\alpha$ needs to be transformed to the discrete time discount factor $0\leq \frac{\alpha}{\alpha+1}<1
$. As we will only work with the discrete model in this paper, we will use $\alpha$ to denote the discrete time discount factor.

For policies $\phi\in \Phi$ such that $\phi(i,1)=\mu_a\in \mathcal{A}(i,1)$, the non-zero transition probabilities 
are given by matrix $P^\phi$ with elements

$$ p^{\phi}_{(i,1),(j,r)} = \begin{cases} \lambda, \   \  
&(j,r)=(i+1,1),\\
\beta,  \ \   &(j,r)=(i,0),\\
\mu_a \mathds{1}_{\{i>0\}},   \ \ &(j,r)=(i-1,1),\\
1-\lambda-\beta-\mu_a\mathds{1}_{\{i>0\}}, \   \ &  (j,r)=(i,1),\end{cases} $$

$$ 
   p^\phi_{(i,0),(j,r)}= \begin{cases} \lambda, \   \ &(j,r)=(i+1,0),\\
\mu\mathds{1}_{\{i>0\}}, \   \ &(j,r)=(i-1,0),\\
1-\lambda-\mu\mathds{1}_{\{i>0\}}, \   \ &(j,r)=(i,0).\end{cases}
 $$ 
The corresponding costs are $c^\phi(i,1)=c_{\mu_a}+h(i)$ and $c^\phi(i,0)=c_\mu+h(i)$.


For this discrete time MDP, we can denote the expected total  discounted saved cost (by the transformed discount factor $0\leq \alpha<1$) under any policy $\phi\in \Phi$ as
 
\begin{equation}\label{questionMathematicallyDiscounted2}
    s^\phi_{\alpha}(p):=\mathbb{E}^\phi_{p^0 \times p^1} \Bigg[ \sum_{n=0}^\infty (1-\alpha)^n\Big( c(X_n)- c(Y_n)\Big) \Bigg],\end{equation}
provided that this expression is well-defined. As we will show in this paper, $s^\phi_\alpha(p)$ exists for the considered probability distributions $p$ and $0\leq \alpha <1$. 
Hence, by Abel's theorem $\lim_{\alpha \downarrow 0}  s^\phi_{\alpha}(p) = s^\phi_0(p)$. 
A priori, the equivalence of Equations~(\ref{questionMathematicallyDiscounted}) and (\ref{questionMathematicallyDiscounted2}) does not need to hold for zero discount factors but follows from the derivation in Appendix~\ref{App:MuEE}.
For the non-discounted saved costs $s_0^\phi(p),s_0(p)$ we drop the subscript 0.

The optimal expected total discounted saved cost is then given by 
     $s_{\alpha}(p) = \sup_{\phi \in \Phi}  s^\phi_{\alpha}(p)$.
These expected saved costs are not affected by uniformisation \cite{serfozo}.
For notational convenience, we denote
$s_\alpha(i):=s_\alpha(\delta_i)$. 




\section{Value iteration and structures of optimal policies}
\label{sect:VIpol}

In this section, we first give the implementation of the Value Iteration algorithm for our problem and show that VI convergences.
For the considered model, we follow Blok and Spieksma~\cite{mdppractice},  where convergence of VI follows from certain drift conditions. 

The convergence of VI allows us to study the structure of optimal policies through propagation of these properties in VI. For this matter, EBDP~\cite{KooleLong} is a useful tool.

For the discounted model, the convergence directly leads to an $\epsilon$-approximation of the saved cost, when starting in specific starting states. This is not the case for the non-discounted model.

The Value Iteration algorithm for our problem with $0\leq \alpha < 1$ is given below (with the convention that $i^+=\max\{i,0\}$ for $i\in \mathbb{Z}$).

\begin{algorithm}[H]
\begin{algorithmic}[1]
\STATE Set $V_{\alpha,0}\equiv 0$ and $n=0$.
\STATE For each $(i,1)\in S$, 
\begin{align*}
     V_{\alpha,n+1}(i,1)  = 
      h(i) +(1-\alpha) \Big(\mu_1V_{\alpha,n}(i,1)+\mu_1 V_{\alpha,n}((i-1)^+,1)+\lambda  V_{\alpha,n}(i+1,0)+\beta V_{\alpha,n}(i,0)\Big)\\
      +  \min_{a\in \{\mu_1,\mu_2\}} \Big\{ c_a +(1-\alpha)\Big ( (a-\mu_1) V_{\alpha,n}((i-1)^+,1)+(\mu_2-a) V_{\alpha,n}(i,1)    \Big)\Big\}.
\end{align*}
$$\phi_{\alpha,n+1}(i,1)=\argmin_{a\in \{\mu_1,\mu_2\}} \Big\{ c_a +(1-\alpha)\Big ( (a-\mu_1) V_{\alpha,n}((i-1)^+,1)+(\mu_2-a) V_{\alpha,n}(i,1)    \Big)\Big\}. $$
For $(i,0)\in S$, 
\begin{align*}
     V_{\alpha,n+1}(i,0) &=  c_\mu+ h(i)\\
     &\quad+(1-\alpha)\Big(\mu V_{\alpha,n}((i-1)^+,0)+(1-\mu- \lambda) V_{\alpha,n}(i,0)
     +\lambda V_{\alpha,n}(i+1,0)\Big) .
\end{align*} 
\STATE Increment $n$ by 1 and return to step 2.
\end{algorithmic}
\caption{Value Iteration for the M/M/1 queue with one single 
control period 
}
\label{VIalpha}
\end{algorithm}
\noindent The calculation of $V_{\alpha,n+1}(i,1)$ 
is given in the most amenable form for our analysis. We add the convention that we drop the subscript 0 for $\phi_{0,n}$ and $V_{0,n}$.

Note that action $\phi_{\alpha,n+1}=\mu_1$  iff \begin{equation}
    \begin{split} \label{eq:convthresh}
           \mu_1(1-\alpha)(V_{\alpha,n}((i-1)^+,1)-V_{\alpha,n}(i,1)) &\leq c_{\mu_2}+ \mu_2(1-\alpha)(V_{\alpha,n}((i-1)^+,1)-V_{\alpha,n}(i,1)) \\
     \Longleftrightarrow V_{\alpha,n}(i,1)-V_{\alpha,n}((i-1)^+,1) &\leq \frac{c_{\mu_2}}{(1-\alpha)(\mu_2-\mu_1)}.
    \end{split}
\end{equation}

For the discounted variant of the MDP ($\alpha>0$) we study the \textit{expected total $\alpha$-discounted cost value function} $V^\phi_\alpha$ using policy $\phi\in \Phi$.
Starting in state $(i,q)\in S$ this gives value
\begin{equation*}
    V^\phi_{\alpha}(i,q)= \mathbb{E}_{(i,q)}^\phi \Big[ \sum_{n=0}^\infty (1-\alpha)^n c(Y_n) \Big].
\end{equation*}
The \textit{minimum expected total
$\alpha$-discounted cost} $V^\ast_\alpha$ in $\Phi$ is defined as $V^\ast_\alpha(i,q)=\inf_{\phi \in \Phi}V^\phi_\alpha(i,q)$.
We note that, in this discounted case, $s^\phi_\alpha(p)=
(p^0-p^1)^TV^\phi_\alpha$, treating $p^0,p^1 $ and $V^\phi_\alpha$ as (infinite) matrices. 
Convergence of $V_{\alpha,n}(i,0)$ to $V_{\alpha}(i,0)$ as $n\rightarrow \infty$ under Assumption~\ref{holdingassump0} has been proved by Lippman~\cite{lippman}.
We will now proceed to determine the general convergence of VI for the discounted and non-discounted models.

\subsection{Convergence of Value Iteration in the discounted case}\label{sect:ConvVI}

In this section, we first give the definitions and assumptions that imply convergence of VI.


\begin{definition}\label{DriftDef}\cite{mdppractice}
    For $\gamma>0$, the function $M: S\rightarrow (0,\infty)$ is called a $(\gamma,\Phi)$-drift function if $P^\phi M \leq \gamma M$ for all $\phi \in \Phi$, where `$\leq$' is the component-wise ordering.
\end{definition}

\begin{definition}\label{MboundDef}
    The Banach space of $M$-bounded functions on $S$ is denoted by $\ell^\infty(S,M)$. That means that $f\in \ell^\infty(S,M)$ if and only if $f:S\rightarrow \mathbb{R}$ and 
    $$ ||f||_M :=\sup_{(i,q)\in S} \frac{|f(i,q)|}{M(i,q)} <\infty.$$
\end{definition}

In addition to Assumption~\ref{holdingassump0}, we impose the following assumption on the holding costs that ensures the existence of policies with a finite average expected cost. Furthermore, this assumption will also help to analyse the discounted model.
\begin{assumption}\label{holdingassump}
  $$  \sum_{i=0}^\infty h(i)\cdot x^i <\infty, \ \ \ \ \text{ for all } 0<x<1.$$
\end{assumption}
Note that this assumption holds for any polynomial function, and more strictly, any function of order $\mathcal{O}(2^{i^\epsilon})$ with $0<\epsilon<1$, see \cite{Kaliski2011}. Also, $h(i)=o(\theta^i)$ for any $\theta>1$, since Assumption~\ref{holdingassump} implies that $ \lim_{i\rightarrow \infty} h(i) \cdot x^i=0\  \forall x\in(0,1)$.
Assumption~\ref{holdingassump} is sufficient to guarantee the average expected cost for the model to be finite. This follows as the average cost for the positive recurrent class $\{(i,0) :  i\in \mathbb{N}_0\}$ is equal to $c_\mu+(1-
\lambda/\mu)\sum_{i=0}^\infty h(i)\cdot (\lambda/\mu
)^i$. 

Convergence of VI for the discounted model is guaranteed by the following assumption (cf. \cite{lippman,Wessels}). 

\begin{assumption}\label{assump2alpha}
\leavevmode
\\
     There exists a constant $1<\gamma<1/(1-\alpha)$ and a $(\gamma,\Phi)$-drift function $M$ such that
    $ \sup_{\phi}||c^\phi||_M<\infty$.

\end{assumption}

For any $1<\gamma<1/(1-\alpha)$ we can define $M$ by $M(i)=M(i,0)=M(i,1)=\gamma^i$. 
Then, for all $\phi\in \Phi$, it is easily checked that
$P^\phi M(i,q)\leq \gamma M(i,q)$. 
As a result, $M$ is a $(\gamma,\Phi)$-drift function. 
Next, we note that $|c^\phi(i,1)|,|c^\phi(i,0)| \leq c_{\mu_2}+h(i)$. As $h(i)=o(\gamma^i)$
by Assumption~\ref{holdingassump}, we find that 
$$\sup_\phi ||c^\phi||_M =\sup_{(i,q),\phi} \frac{|c^\phi(i,q)|}{M(i,q)} <\infty.$$


Thus, VI converges such that $\lim_{n\rightarrow \infty} V_{\alpha,n}(i,q)=V_\alpha(i,q)$ and $V_\alpha(i,q)= \min_{\phi \in \Phi}V_\alpha^\phi(i,q)$, for $(i,q)\in S$. Furthermore, $V_\alpha$ is the unique $M$-bounded solution to the Discrete Discounted Optimality Equation (DDOE)
\begin{equation}\label{DDAO}
     u(x) =\inf_{\phi \in \Phi} \left\{ c^\phi(x) + (1-\alpha)\sum_{y\in S}p^\phi_{x,y}u(y) \right\}.
\end{equation}

\subsection{Convergence for the non-discounted model}\label{sect:ConvVInonDisc}

 The  
 Discrete time Average cost Optimality Equation (DAOE) takes a central role in the study of non-discounted MDPs over the infinite time horizon.
 We will show that VI converges to a solution of the DAOE
 given by
\begin{equation}\label{DAOE}
   g+H(x) =\min_{\phi \in \Phi} \left\{ c^\phi(x) + \sum_{y\in S}p^\phi_{x,y}H(y) \right\}  , \ \ \ x\in S,
\end{equation}
with $g\in \mathbb{R}_{+}$ the average expected cost and $\ H:S\rightarrow \mathbb{R}$ a \textit{relative value} vector respectively.

Note that the limiting distribution of our model is the unique stationary distribution of the M/M/1 queue with rate $\mu$. This follows as the service rate control is lost after a finite expected time
after which the process behaves as an M/M/1 queue. Thus, the average expected cost of the model is equal to the stationary expected cost of the M/M/1 queue.
The stationary distribution $\pi^0_\mu$ of the number of customers in the system 
is given by 
$\pi^0_\mu(i,0)= (1-\rho_\mu)\rho_\mu^i, \pi^0_\mu(i,1)=0$, where $\rho_\mu=
\lambda/\mu, i\in \mathbb{N}_0$. 
The average expected cost $g_\mu$ per time unit is given by 
\begin{equation} \label{eq:gf}
    g_\mu= c_\mu+\sum_{i=0}^\infty h(i)\cdot \pi^0_\mu(i,0) .\end{equation}


 For the fixed rate M/M/1 queue we note that the DAOE is given by tuples $(g_\mu,H_{\mu})$ with $g_\mu\in \mathbb{R}_{+},$ \\
 $H_{\mu}:\mathbb{N}_0\rightarrow \mathbb{R}$, such that
\begin{equation}\label{CAOEf}
    g_\mu +H_{\mu}(i) =  c_\mu+h(i) + \mu H_{\mu}((i-1)^+)+(1-\lambda-\mu)H_{\mu}(i)  + \lambda H_{\mu}(i+1), \ \quad i\in \mathbb{N}_0.
\end{equation}

The strong link of our model to this DAOE is that $g=g_\mu,$ $H(i,0)=H_{\mu}(i)$. 

We can 
derive these values exactly from the optimality equation.
 First, w.l.o.g., $H_{\mu}(0)=0$, then $H_{\mu}(1)=\frac{g_\mu-c_\mu-h(0)}{\lambda}=\frac{g_\mu-c_\mu}{\lambda}$, all the other values of $H_{\mu}(i)$ can iteratively be determined as \begin{equation}\label{Hiter} H_{\mu}(i+1)= H_{\mu}(i)+\frac{g_\mu-c_\mu-h(i)+\mu(H_{\mu}(i)-H_{\mu}(i-1))}{\lambda}. \end{equation}

As this M/M/1 queue is a special case of the service rate control model studied in \cite{lippman}, 
the value function is
non-decreasing and convex. Hence, the relative value function $H_{\mu}$ is non-decreasing and convex for the M/M/1 queue with convex holding cost function $h$. 

In order to conclude convergence of VI to determine a solution of the DAOE that includes the transient states $(i,1)\in S$, we use the following theorem from Spieksma~\cite{spieksma1990geometrically} and Altman et al.~\cite{Altman} as taken from \cite[Thm 5.8]{mdppractice}.

\begin{restatable}{theorem}{convergenceVIM}\label{convergenceVIM} 
    Suppose 
that the MDP is $M$-uniformly geometricly recurrent, i.e., there exist
a function \\
$M :  S \rightarrow [1,\infty)$, a finite set $D\subset S$ and a constant  $\eta<1$, such that
\begin{equation}\label{eq:M}
    \sum_{y\notin D } p^\phi_{xy}M(y)\leq \eta M(x), \quad x\in S, \phi\in \Phi. 
\end{equation} 
Suppose further that $\sup_{\phi \in \Phi} ||c^\phi||_M<\infty$ and that, under any policy $\phi\in \Phi$ the resulting Markov process is aperiodic and has one closed
class, plus possibly transient states.
Let $0 \in S$ be a selected state. Then, there exists a unique solution pair  $(g^\ast,H^\ast)$  with
$H^\ast \in l^\infty(S,M) $, and $H^\ast(0) = 0$, to the DAOE (Equation~(\ref{DAOE})) with the properties that 
\begin{enumerate}
    \item $g^\ast$   is the minimum average expected cost (in $\Phi$),
    \item any $\phi^\ast\in \Phi$ with $\phi^\ast \in \argmin_{\phi \in \Phi} \left\{ c^\phi(x) + \sum_{y\in S}p^\phi_{x,y}H(y) \right\}$ is (average cost) optimal in $\Phi$ and
    \item there exists $x^\ast\in D$ with $H^\ast(x^\ast) = \inf_xH^\ast(x)$ .
\end{enumerate}
Furthermore, average cost VI converges, that is, $lim_{n\rightarrow \infty}(V_n-V_n(0)\boldsymbol{e}) = H^\ast$, and
any limit point of the sequence $ \{\phi_n\}_n$ is average cost optimal and a minimiser of the
DAOE (\ref{DAOE}) with solution tuple $ (g^\ast,H^\ast)$.
\end{restatable}

We can verify that Theorem~\ref{convergenceVIM} holds for our model with a certain choice of $M$ and selected state $0=(0,0)$.
For $\zeta \in \mathbb{R}_{\geq 1}$ we define $M$ by $$M(i,0)=\frac{1}{2} \cdot \zeta^i, M(i,1)= \zeta^i,$$ and take $D=\{(0,0)\}$.

For $\mu_1>\lambda$, Equation~(\ref{eq:M}) holds for $1<\zeta < \frac{\mu_1}{\lambda}\wedge \big(1+\frac{\beta}{\lambda}\big),$ with 
$$\eta= 1-\lambda+\lambda\zeta+\max\Big\{-\beta, -(1-\frac{1}{\zeta})\mu, -(1-\frac{1}{\zeta})\mu_1 -\frac{1}{2}\beta \Big\}  <1.$$

For $\mu_1\leq \lambda < \mu_2$ we have to restrict $\zeta$ to $$1<\zeta< \frac{\lambda+\beta}{\lambda}\wedge \frac{\mu}{\lambda} \wedge\frac{(\lambda+\frac{1}{2}\beta+\mu_1)+\sqrt{(\lambda+\frac{1}{2}\beta+\mu_1)^2-4\lambda \mu_1}}{2\lambda}$$ with resulting $$
        \eta = \max\Big\{1-\lambda+\lambda\zeta+\max\Big\{-\beta, -(1-\frac{1}{\zeta})\mu, -(1-\frac{1}{\zeta})\mu_1 -\frac{1}{2}\beta ,\frac{\mu_1}{\zeta}+(1-\lambda-\frac{1}{2}\beta-\mu_1)+\lambda \zeta\Big\} <1.$$
In Appendix~\ref{app:Mrec} we show that Equation~(\ref{eq:M}) holds for these choices of $M$, $\eta$ and $D$. 
Furthermore, we require $\zeta <1/(1-\beta)$ for Lemma~\ref{HAppr} to hold. This requirement is not needed for alternative $\epsilon$-approximations of $H^\ast(0,1),H^\ast(1,1),\dots,H^\ast(k,1)$ given in Appendix~\ref{sect:MProofs}.

Note that $$\sup_\phi ||c^\phi||_M =\sup_{(i,q),\phi} \frac{|c^\phi(i,q)|}{M(i,q)} \leq  \sup_{(i,q)} \frac{|c_{\mu_2}+h(i)|}{M(i,q)}\leq c_{\mu_2}+\sup_i \frac{2h(i)}{\zeta^i}$$ is finite by Assumption~\ref{holdingassump}.


Finally, the Markov process under any policy $\phi\in \Phi$ is indeed aperiodic and has one closed
class $\{(i,0) :  i\in \mathbb{N}_0\}$, plus transient states $\{(i,1) :  i\in \mathbb{N}_0\}$.
Thus, we conclude that Theorem~\ref{convergenceVIM} is applicable to our problem.


\subsection{The structure of policies in Value Iteration }
\label{sect:ebdp}

Our aim is to show that the optimal policy has a threshold structure on $\{(i,1)\ : \ i\in \mathbb{N}_0\}$. To this end, it is sufficient to show convexity of $V_{\alpha,n}(i,0)$ and $V_{\alpha,n}(i,1)$ in $i$ using the methodology of EBDP by Koole~\cite{KooleLong}.
Service rate control is not directly included in \cite{KooleLong}, but combining the dynamics of controlled arrivals and of departures allows to deduce the necessary structural results.

\begin{proposition} \label{increasingconvexvalues}
    Let $\alpha\in [0,1), V_{\alpha,0}\equiv 0$ and $n\in \mathbb{N}_0$.
    Then, $V_{\alpha,n}(i,0)$ and $V_{\alpha,n}(i,1)$ are non-decreasing and convex in $i\in \mathbb{N}_0$. 
    The policy $\phi_{\alpha,n}\in \Phi$ 
    is of threshold form. I.e., there exists a threshold\\
    $i_{\alpha,n}\in \mathbb{N}_0\cup\{\infty\}$, such that $$\phi_{\alpha,n}(i,1) =\begin{cases}
        \mu_1 \ \ \ \ \ \text{ if } \ \ i\leq i_{\alpha,n},\\
         \mu_2 \ \ \ \ \ \text{ if } \ \ i> i_{\alpha,n}.
    \end{cases}$$
\end{proposition}
\begin{proof}
For this proof, we take $g_1,g_2,\hdots,g_5:\mathbb{N}_0\rightarrow \mathbb{R}$ to be general functions upon which the following operators act.
 For notational convenience,
 denote 
 $V_{\alpha,n}^q(i)=V_{\alpha,n}(i,q)$ for $q\in \{0,1\},i\in \mathbb{N}_0$.
Define the following operators

    \begin{itemize}
        \item $T_{\text{disc}}\ g_1(i)=h(i) +(1-\alpha)g_1(i)$.
        \item $T^\mu_{\text{disc}}\ g_1(i)=c_\mu+ h(i) +(1-\alpha)g_1(i)$.
        \item $T_A\ g_1(i)=g_1(i+1). $
        \item $T_D\ g_1(i)= g_1((i-1)^+)$.
        \item $ T_{CA}\ g_1(i) = \min\Big\{ \frac{c_{\mu_2}}{(1-\alpha)(\mu_2-\mu_1)} + g_1(i),g_1(i+1)\Big\}$.
        \item $ T_{\text{unif}}^0(g_1,g_2,g_3)(i)=\lambda g_1(i) + (1-\lambda-\mu) g_2(i)+\mu g_3(i).$ 
           \item $ T_{\text{unif}}^1(g_1,g_2,g_3,g_4,g_5)(i)=\lambda g_1(i) + \beta g_2(i)+\mu_1 g_3(i)+\mu_1 g_4(i) +(\mu_2-\mu_1)g_5(i) $.
    \end{itemize}
    Note that $T_D (T_{CA}\ g(i))=\min\Big\{ g(i),\frac{c2}{(1-\alpha)(\mu_2-\mu_1)} +g((i-1)^+)\Big\}$.   
This leads to
\begin{equation*}
\begin{split}
V_{\alpha,n+1}(i,0)
&=c_\mu+h(i)\\
&\quad +(1-\alpha) \Big(\lambda  V_{\alpha,n}(i+1,0)
     + (1-\lambda-\mu)V_{\alpha,n}(i,0)+\mu V_{\alpha,n}((i-1)^+,0) \Big)\\
     &=T^\mu_{\text{disc}}( T_{\text{unif}}^0 (T_A V_{\alpha,n}^0,V_{\alpha,n}^0,T_D V_{\alpha,n}^0 ) )(i).
\end{split}
\end{equation*}
Furthermore, as $\lambda+\beta+\mu_1+\mu_2=1$,
 \begin{align*}
     V_{\alpha,n+1}(i,1)&= 
     h(i) +(1-\alpha) \Big(\mu_1V_{\alpha,n}(i,1)+\mu_1 V_{\alpha,n}((i-1)^+,1)+\lambda  V_{\alpha,n}(i+1,1)+\beta V_{\alpha,n}(i,0)\Big)\\
      &\hspace{2.3cm}+ \min_{a\in \{\mu_1,\mu_2\}} \Big\{ c_a +(1-\alpha)\Big ( (a-\mu_1) V_{\alpha,n}((i-1)^+,1)+(\mu_2-a) V_{\alpha,n}(i,1)    \Big)\Big\}    \\
     &=h(i) +(1-\alpha) \Big(\lambda  V_{\alpha,n}(i+1,1)+\beta V_{\alpha,n}(i,0)
     + \mu_1V_{\alpha,n}(i,1)+\mu_1 V_{\alpha,n}((i-1)^+,1)\\
     & \hspace{3.3cm} +(\mu_2-\mu_1)\min\Big\{ V_{\alpha,n}(i,1), \frac{c_{\mu_2}}{(1-\alpha)(\mu_2-\mu_1)} + V_{\alpha,n}((i-1)^+,1)\Big\} \Big) \\
     &=T_{\text{disc}}( T_{\text{unif}}^1 (T_A V_{\alpha,n}^1,V_{\alpha,n}^0, V_{\alpha,n}^1 ,T_D V_{\alpha,n}^1, T_D (T_{CA} V_{\alpha,n}^1 ) ) )(i).
\end{align*}

Observe that $V_{\alpha,n+1}^0$ is the result of a combination of operators acting on $ V_{\alpha,n}^0$ but $V_{\alpha,n+1}^1$ is the result of a combination of operators acting on both $ V_{\alpha,n}^0$ and $  V_{\alpha,n}^1$ .
According to Koole~\cite{KooleLong}, for non-decreasing and convex $h$, all operators $T_{\text{disc}}$ $,T^ f_{\text{disc}}$ $,T_{\text{unif}}^0$, $T_D, $ $T_{CA},$ $T_A$, and $T_{\text{unif}}^1$, map non-decreasing, convex functions to other non-decreasing, convex functions. 
As a result, both $V_{\alpha,n}(i,0)$ and $V_{\alpha,n}(i,1)$ are non-decreasing and convex for any $n\in \mathbb{N}_0$ (and $0\leq \alpha<1$) by induction. 

Threshold optimality now follows directly from Equation~(\ref{eq:convthresh}). This shows the actions taken by $\phi_{\alpha,n}$ for states $(i,1)$
are of a threshold form in $i$.

\end{proof}

Next, we derive monotonicity properties of the thresholds $i_{\alpha,n}$ by analysing properties of $V_{\alpha,n}$.


\begin{proposition} \label{decreasingthreshold}
    Let  $\alpha\in[0,1),V_{\alpha,0}\equiv 0$. 
    For $i\in \mathbb{N}_{\geq 1},q\in \{0,1\}$,
    \begin{itemize}
        \item  $V_{\alpha,n}(i,q)-V_{\alpha,n}(i-1,q)$ is non-decreasing in $n\in \mathbb{N}_0$,
        \item  $V_{\alpha,n}(i,q)-V_{\alpha,n}(i-1,q)$ is non-increasing in $\alpha$.
    \end{itemize}
     As a result, the thresholds $i_{\alpha,n}$ are non-increasing in $n$ and non-decreasing in $\alpha$.
\end{proposition}
\begin{proof}
We prove these statements by induction in $n$, starting with monotonicity in $\alpha$.

    Let $\alpha'>\alpha$, $i\in \mathbb{N}_{\geq 1},q\in \{0,1\}$.    We consider $V_{\alpha,n}(i,q)-V_{\alpha,n}(i-1,q)$. Clearly,\\
    $V_{\alpha,0}(i,q)-V_{\alpha,0}(i-1,q)=0$ is non-increasing in $\alpha$.
    Let $n\in \mathbb{N}_0$.
    Assume that $V_{\alpha,n}(i,q)-V_{\alpha,n}(i-1,q)$ is non-increasing in $\alpha$ for $i\in \mathbb{N}_{\geq 1},q\in \{0,1\}$. 
    The induction hypothesis combined with Equation~(\ref{eq:convthresh}) implies that threshold $i_{\alpha,n+1}$ is non-decreasing in $\alpha$ .
    The induction hypothesis and the fact that $V_{\alpha,n}(j,0)$ is non-decreasing in $j$ (cf. Proposition~\ref{increasingconvexvalues}) give that
    \begin{equation*}
        \begin{split}
           & V_{\alpha',n+1}(i,0)-V_{\alpha',n+1}(i-1,0)=h(i)-h(i-1)\\
            &\quad \quad+(1-\alpha')\Big(\mu (V_{\alpha',n}(i-1,0)-V_{\alpha',n}((i-2)^+,0))\\
            &\quad \quad \quad \quad \quad \quad\quad+(1-\mu- \lambda) (V_{\alpha',n}(i,0)-V_{\alpha',n}(i-1,0))
     +\lambda (V_{\alpha',n}(i+1,0)-V_{\alpha',n}(i,0))\Big) \\
       &\quad \leq h(i)-h(i-1) \\
     &\quad \quad  + (1-\alpha)\Big(\mu (V_{\alpha,n}(i-1,0)-V_{\alpha,n}((i-2)^+,0))\\
      &\quad \quad \quad \quad \quad \quad\quad+(1-\mu- \lambda) (V_{\alpha,n}(i,0)-V_{\alpha,n}(i-1,0))
     +\lambda (V_{\alpha,n}(i+1,0)-V_{\alpha,n}(i,0))\Big) \\
      &\quad = V_{\alpha,n+1}(i,0)-V_{\alpha,n+1}(i-1,0).
        \end{split}
    \end{equation*}

Next, we verify that $V_{\alpha',n+1}(i,1)-V_{\alpha',n+1}(i-1,1) \leq V_{\alpha,n+1}(i,1)-V_{\alpha,n+1}(i-1,1)$. Now,

\begin{equation*}
        \begin{split}
           & V_{\alpha',n+1}(i,1)-V_{\alpha',n+1}(i-1,1)=h(i)-h(i-1)\\
            &\quad \quad+(1-\alpha') \Big(\mu_1(V_{\alpha',n}(i,1)-V_{\alpha',n}(i-1,1))+\mu_1 (V_{\alpha',n}(i-1,1)-V_{\alpha',n}((i-2)^+,1))\\
            &\quad \quad\quad \quad\quad \quad\quad \quad \quad\quad \quad +\lambda  (V_{\alpha',n}(i+1,0)-V_{\alpha',n}(i,1))+\beta (V_{\alpha',n}(i,0)-V_{\alpha',n}(i-1,0))\Big)\\
      &\quad \quad+  \min_{a\in \{\mu_1,\mu_2\}} \Big\{ c_a +(1-\alpha')\Big ( (a-\mu_1) V_{\alpha',n}(i-1,1)+(\mu_2-a) V_{\alpha',n}(i,1)    \Big)\Big\}\\
       &\quad \quad - \min_{a\in \{\mu_1,\mu_2\}} \Big\{ c_a +(1-\alpha')\Big ( (a-\mu_1) V_{\alpha',n}((i-2)^+,1)+(\mu_2-a) V_{\alpha',n}(i-1,1)    \Big)\Big\}.
        \end{split}
    \end{equation*}
    
    Note that the increments are non-negative by virtue of Proposition~\ref{increasingconvexvalues}.
    The induction hypothesis yields 
    \begin{equation}
        \begin{split}\label{eq:split1}
   &(1-\alpha') \Big(\mu_1(V_{\alpha',n}(i,1)-V_{\alpha',n}(i-1,1))+\mu_1 (V_{\alpha',n}(i-1,1)-V_{\alpha',n}((i-2)^+,1))\\
            & \quad\hspace{3.3cm}          +\lambda  (V_{\alpha',n}(i+1,0)-V_{\alpha',n}(i,1))+\beta (V_{\alpha',n}(i,0)-V_{\alpha',n}(i-1,0))\Big)\\
    &\quad \leq       (1-\alpha) \Big(\mu_1(V_{\alpha,n}(i,1)-V_{\alpha,n}(i-1,1))+\mu_1 (V_{\alpha,n}(i-1,1)-V_{\alpha,n}((i-2)^+,1))\\
            &\quad \hspace{3.5cm}          +\lambda  (V_{\alpha,n}(i+1,0)-V_{\alpha,n}(i,1))+\beta (V_{\alpha,n}(i,0)-V_{\alpha,n}(i-1,0))\Big).
        \end{split}
    \end{equation}

    We now proceed to show that \begin{equation}
        \begin{split} \label{eq:splitmin}
           & \min_{a\in \{\mu_1,\mu_2\}} \Big\{ c_a +(1-\alpha')\Big ( (a-\mu_1) V_{\alpha',n}(i-1,1)+(\mu_2-a) V_{\alpha',n}(i,1)    \Big)\Big\}\\
       & \quad \quad - \min_{a\in \{\mu_1,\mu_2\}} \Big\{ c_a +(1-\alpha')\Big ( (a-\mu_1) V_{\alpha',n}((i-2)^+,1)+(\mu_2-a) V_{\alpha',n}(i-1,1)    \Big)\Big\}\\
       &\quad \leq  \min_{a\in \{\mu_1,\mu_2\}} \Big\{ c_a +(1-\alpha)\Big ( (a-\mu_1) V_{\alpha,n}(i-1,1)+(\mu_2-a) V_{\alpha,n}(i,1)    \Big)\Big\}\\
       &\quad \quad - \min_{a\in \{\mu_1,\mu_2\}} \Big\{ c_a +(1-\alpha)\Big ( (a-\mu_1) V_{\alpha,n}((i-2)^+,1)+(\mu_2-a) V_{\alpha,n}(i-1,1)    \Big)\Big\}.
        \end{split}
    \end{equation} 

To this end, we will consider all possible combinations of minimising actions,
    which are denoted by $a_i',a_{i-1}',a_i$, and $a_{i-1}$, respectively. If $a'_{i-1}=a'_{i}=a_{i-1}=a_{i}$, then we can use non-increasingness of  $V_{\alpha,n}(j,1)-V_{\alpha,n}((j-1)^+,1)$ in $\alpha$ and non-decreasingness of $V_{\alpha,n}(j,1)$ in $j$ to find
 \begin{equation*}
        \begin{split}
           & \min_{a\in \{\mu_1,\mu_2\}} \Big\{ c_a +(1-\alpha')\Big ( (a-\mu_1) V_{\alpha',n}(i-1,1)+(\mu_2-a) V_{\alpha',n}(i,1)    \Big)\Big\}\\
       &  \quad - \min_{a\in \{\mu_1,\mu_2\}} \Big\{ c_a +(1-\alpha')\Big ( (a-\mu_1) V_{\alpha',n}((i-2)^+,1)+(\mu_2-a) V_{\alpha',n}(i-1,1)    \Big)\Big\}\\
       & =  c_{a'_i}+ (1-\alpha')\Big ( ({a'_i}-\mu_1) V_{\alpha',n}(i-1,1)+(\mu_2-{a'_i}) V_{\alpha',n}(i,1)    \Big)\\
       & \quad - c_{a'_{i-1}}- (1-\alpha')\Big ( ({a'_{i-1}}-\mu_1) V_{\alpha',n}((i-2)^+,1)+(\mu_2-{a'_{i-1}}) V_{\alpha',n}(i-1,1)    \Big)\\
       &= (1-\alpha')\Big ( ({a'_i}-\mu_1) (V_{\alpha',n}(i-1,1) -V_{\alpha',n}((i-2)^+,1))+(\mu_2-{a'_i})( V_{\alpha',n}(i,1) -V_{\alpha',n}(i-1,1))\Big) \\
         &\leq (1-\alpha)\Big ( ({a'_i}-\mu_1) (V_{\alpha,n}(i-1,1) -V_{\alpha,n}((i-2)^+,1))+(\mu_2-{a'_i})( V_{\alpha,n}(i,1) -V_{\alpha,n}(i-1,1))\Big) \\
       &=  \min_{a\in \{\mu_1,\mu_2\}} \Big\{ c_a +(1-\alpha)\Big ( (a-\mu_1) V_{\alpha,n}(i-1,1)+(\mu_2-a) V_{\alpha,n}(i,1)    \Big)\Big\}\\
       & \quad - \min_{a\in \{\mu_1,\mu_2\}} \Big\{ c_a +(1-\alpha)\Big ( (a-\mu_1) V_{\alpha,n}((i-2)^+,1)+(\mu_2-a) V_{\alpha,n}(i-1,1)    \Big)\Big\}.
        \end{split}
    \end{equation*} 

    If $(a_{i-1}',a_{i}')=(a_{i-1},a_{i})=(\mu_1,\mu_2)$,
    \begin{equation*}
        \begin{split}
           & \min_{a\in \{\mu_1,\mu_2\}} \Big\{ c_a +(1-\alpha')\Big ( (a-\mu_1) V_{\alpha',n}(i-1,1)+(\mu_2-a) V_{\alpha',n}(i,1)    \Big)\Big\}\\
       &  \quad - \min_{a\in \{\mu_1,\mu_2\}} \Big\{ c_a +(1-\alpha')\Big ( (a-\mu_1) V_{\alpha',n}((i-2)^+,1)+(\mu_2-a) V_{\alpha',n}(i-1,1)    \Big)\Big\}\\
       & =  c_{\mu_2}+ (1-\alpha')(\mu_{2}-\mu_1) V_{\alpha',n}(i-1,1) -  (1-\alpha')(\mu_2-\mu_{1}) V_{\alpha',n}(i-1,1) =c_{\mu_2}   \\
       &=  \min_{a\in \{\mu_1,\mu_2\}} \Big\{ c_a +(1-\alpha)\Big ( (a-\mu_1) V_{\alpha,n}(i-1,1)+(\mu_2-a) V_{\alpha,n}(i,1)    \Big)\Big\}\\
       & \quad - \min_{a\in \{\mu_1,\mu_2\}} \Big\{ c_a +(1-\alpha)\Big ( (a-\mu_1) V_{\alpha,n}((i-2)^+,1)+(\mu_2-a) V_{\alpha,n}(i-1,1)    \Big)\Big\}.
        \end{split}
    \end{equation*} 

    We note that $$a'_j=\mu_2\implies V_{\alpha',n}(j,1)-V_{\alpha',n}(j-1,1)\geq \frac{c_{\mu_2}}{(1-\alpha')(\mu_2-\mu_1)},$$ and therefore $$V_{\alpha,n}(j,1)-V_{\alpha,n}(j-1,1)\geq V_{\alpha',n}(j,1)-V_{\alpha',n}(j-1,1)\geq \frac{c_{\mu_2}}{(1-\alpha')(\mu_2-\mu_1)},$$ as such  $a'_j=\mu_2$ implies $a_j=\mu_2$. Furthermore, the optimal actions are of a threshold form by Proposition~\ref{increasingconvexvalues}.
 Thus, it only remains to check the cases
 $$ (a_{i-1}',a_i',a_{i-1},a_i)\in\{
     (\mu_1,\mu_1,\mu_1,\mu_2),
     (\mu_1,\mu_1,\mu_2,\mu_2),
     (\mu_1,\mu_2,\mu_2,\mu_2).
 \}$$

We note that the left-hand side in Equation~(\ref{eq:splitmin}) increases if we change the optimal action $a_i'$. Likewise, the right-hand side in Equation~(\ref{eq:splitmin}) decreases if we change the optimal action $a_{i-1}$. This means that Equation~(\ref{eq:splitmin}) holds for the three remaining cases through reduction to the case $(a_{i-1}',a_i')=(\mu_1,\mu_2),\\
(a_{i-1},a_i)=(\mu_1,\mu_2)$.

    Now, we can conclude that in any of the cases, Equation~(\ref{eq:splitmin}) holds and together with Equation~(\ref{eq:split1}) we can conclude that $V_{\alpha',n+1}(i,1)-V_{\alpha',n+1}(i-1,1)\leq V_{\alpha,n+1}(i,1)-V_{\alpha,n+1}(i-1,1) $.
    Again, by Equation~(\ref{eq:convthresh}), this implies that $i_{\alpha,n+2}$ is non-decreasing in $\alpha$.
    This completes the induction step.


The facts that $V_{\alpha,n}(i,q)-V_{\alpha,n}(i-1,q)$ is non-decreasing in $n\in \mathbb{N}_0$ for any $i\in \mathbb{N}_{\geq 1},q\in \{0,1\}$ and that $i_{\alpha,n}$ is non-increasing in $n$, can be shown analogously.

\end{proof}

Note that $i_{\alpha,0}=\infty$ for any $0\leq \alpha<1$. However, 
if $i_\alpha^\ast<\infty$,
 we can determine an $n\in \mathbb{N}_0$ such that $i_{\alpha,n}$ is finite. To this end, we first derive suitable bounds on $V_{\alpha,n}(i,1)-V_{\alpha,n}(i-1,1)$ based on the convex holding cost function. 

\begin{lemma} \label{boundedthresholdhelp}
     Let $\alpha\in (0,1)$, $V_{\alpha,0}\equiv 0$, and suppose that $i,n\in \mathbb{N}_0$ are given such that $i>n$ and $i_{\alpha,n}\geq i+2n$.
  Then,
  \begin{equation*}
      \begin{split}
          \frac{1-(1-\alpha)^n}{\alpha}\cdot (h(i-n+1)-h(i-n)) &\leq V_{\alpha,n}(i,1)-V_{\alpha,n}(i-1,1) \\
          &\leq \frac{1-(1-\alpha)^n}{\alpha}\cdot (h(i+n)-h(i+n-1)).
      \end{split}
  \end{equation*}
  For $\alpha=0$, let $i,n\in \mathbb{N}_0$ be given such that $i>n$ and $i_{n}\geq i+2n$. Then, 
    \begin{equation*}
      \begin{split}
          n (h(i-n+1)-h(i-n)) &\leq V_{n}(i,1)-V_{n}(i-1,1) \leq n (h(i+n)-h(i+n-1)).
      \end{split}
  \end{equation*}
\end{lemma}
\begin{proof}
    Let $n,i\in \mathbb{N}_0$ 
    such that $i>n$ and $i_{\alpha,n}\geq i+2n$. Proposition~\ref{decreasingthreshold} implies that  $i_{\alpha,m}\geq i+2n$ 
     for $0\leq m \leq n$. 

    By induction, we will now show that \begin{equation*}
        \begin{split}
             \frac{1-(1-\alpha)^m}{\alpha}\cdot (h(j-m+1)-h(j-m)) &\leq V_{\alpha,m}(j,q)-V_{\alpha,m}(j-1,q)\\
             &\leq \frac{1-(1-\alpha)^m}{\alpha}\cdot (h(j+m)-h(j+m-1))
        \end{split}
    \end{equation*} for $0\leq m\leq n,q\in \{0,1\}$ and  $i-n\leq j-m \leq i+n $.

    As $V_{\alpha,0}\equiv 0$, the statement holds for $m=0$.
     Suppose that the statement holds for $m <n$. We show that the statement also holds for $m+1$.
     Let $i-n\leq j-m-1 \leq i+n $.
     As $i_{\alpha,m+1}\geq i+2n \geq j,j-1$, the optimal action in step $m+1$ of VI in states $(j,1)$ and $(j-1,1)$ is to use service rate $\mu_1$. Thus,
\begin{equation*}
    \begin{split}
        &V_{\alpha,m+1}(j,1)-V_{\alpha,m+1}(j-1,1)=h(j)-h(j-1)\\
        & \quad \quad \quad  +(1-\alpha)\Big(\mu_1( V_{\alpha,m}(j-1,1)-V_{\alpha,m}(j-2,1)) +\mu_2( V_{\alpha,m}(j,1)-V_{\alpha,m}(j-1,1))\\
        &\hspace{3.7cm}+\beta ( V_{\alpha,m}(j,0)-V_{\alpha,m}(j-1,0))
        +\lambda( V_{\alpha,m}(j+1,1)-V_{\alpha,m}(j,1))\Big).
    \end{split}
\end{equation*}

By convexity of $V_{\alpha,m}(i,q)$ in $i$ (cf. Proposition~\ref{increasingconvexvalues}) this leads to 
\begin{equation}
    \begin{split}\label{eq:insluiten}
        &h(j)-h(j-1)\\
        & \quad \ +(1-\alpha)\Big((1-\beta)( V_{\alpha,m}(j-1,1)-V_{\alpha,m}(j-2,1))+\beta ( V_{\alpha,m}(j-1,0)-V_{\alpha,m}(j-2,0)) \Big)\\
     & \ \ \leq V_{\alpha,m+1}(j,1)-V_{\alpha,m+1}(j-1,1)\\
      &\ \ \leq h(j)-h(j-1)\\
      & \quad \ +(1-\alpha)\Big((1-\beta)( V_{\alpha,m}(j+1,1)-V_{\alpha,m}(j,1))+\beta ( V_{\alpha,m}(j+1,0)-V_{\alpha,m}(j,0))\Big).
    \end{split}
\end{equation}
 Plugging the induction hypothesis for $m$ into Equation~(\ref{eq:insluiten}), using convexity of $h$ and the fact that

\begin{equation*}    \begin{split}\label{eq:inperkhulp}
        1+(1-\alpha)\cdot \frac{1-(1-\alpha)^m}{\alpha} 
       = \frac{1-(1-\alpha)^{m+1}}{\alpha}.
    \end{split}
\end{equation*}
 yields that
 \begin{equation*}
\begin{split}\label{eq:insluiten2}
        &h(j-m)-h(j-m-1)+(1-\alpha)\Big(\frac{1-(1-\alpha)^m}{\alpha}\cdot (h(j-m)-h(j-m-1))\Big)\\
        &=\frac{1-(1-\alpha)^{m+1}}{\alpha}\cdot (h(j-(m+1)+1)-h(j-(m+1))) 
        \\
     &  \leq V_{\alpha,m+1}(j,1)-V_{\alpha,m+1}(j-1,1)\\
      &\leq h(j+1+m)-h(j+m) +(1-\alpha)\Big(\frac{1-(1-\alpha)^m}{\alpha}\cdot (h(j+1+m)-h(j+m))\Big)\\
      &= \frac{1-(1-\alpha)^{m+1}}{\alpha}\cdot (h(j+(m+1))-h(j+(m+1)-1)).
    \end{split}
\end{equation*}

Likewise, one can also show that the induction step holds for the case $q=0$, which concludes the proof.

The non-discounted case ($\alpha=0$) can be shown analogously.
\end{proof}

\begin{proposition}
    \label{boundedthreshold}
     Let $ \alpha\in(0,1)$, $V_{\alpha,0}\equiv 0$ and let $$\sup_i (h(i)-h(i-1))\leq \frac{\alpha c_{\mu_2}}{(1-\alpha)(\mu_2-\mu_1)},$$ then $i_{\alpha,n}=\infty$ for all $n\geq 1
     $. If 
     $$\sup_i (h(i)-h(i-1))> \frac{\alpha c_{\mu_2}}{(1-\alpha)(\mu_2-\mu_1)},$$ then we can compute a value of $n\geq 1$ and $B_n\in \mathbb{N}_0$ such that $i_{\alpha,n}<B_n$.

     For the non-discounted case, we can compute
     $n\geq 1$ and $B_n\in \mathbb{N}_0$ such that $i_n\leq B_n$.
\end{proposition} 
\begin{proof}
    Let 
    $\alpha\in(0,1)$, and let $$\sup_i (h(i)-h(i-1))\leq \frac{\alpha c_{\mu_2}}{(1-\alpha)(\mu_2-\mu_1)}.$$ 
    Clearly, $i_{\alpha,0}=\infty$ by Equation~(\ref{eq:convthresh}) as $V_{\alpha,0}\equiv 0$.
    We proceed by induction. Assume that $i_{\alpha,n}=\infty$ for some 
    $n\geq 1$. Applying Lemma~\ref{boundedthresholdhelp} for $i>n$, we obtain
    \begin{equation*}
        \begin{split}
            V_{\alpha,n}(i,1)-V_{\alpha,n}(i-1,1) 
          &\leq \frac{1-(1-\alpha)^n}{\alpha}\cdot (h(i+n)-h(i+n-1)) \\
          &<\frac{1}{\alpha} \cdot \sup_j(h(j)-h(j-1))\leq \frac{c_{\mu_2}}{(1-\alpha)(\mu_2-\mu_1)}.
        \end{split}
    \end{equation*}

Thus, for any state $(i,1)$ with $i>n$, action $\mu_1$ is optimal in step $n+1$ of VI, i.e., $i_{\alpha,n+1}=\infty$. 

Next, let $$\sup_i (h(i)-h(i-1))> \frac{\alpha c_{\mu_2}}{(1-\alpha)(\mu_2-\mu_1)}.$$ Then
there exists a $j\geq 2$ such that $$h(j)-h(j-1)> \frac{\alpha c_{\mu_2}}{(1-\alpha)(\mu_2-\mu_1)}.$$ For this  $j$ we choose $n$, such that also  $$\frac{1-(1-\alpha)^n}{\alpha}\cdot(h(j)-h(j-1))> \frac{\alpha c_{\mu_2}}{(1-\alpha)(\mu_2-\mu_1)}.$$ Next, 
let $i=j+n-1>n$ and $B_n=i+2n$.  Either $i_{\alpha,n}< i+2n\leq B_n$ or $i_{\alpha,n}\geq B_n$. In the second case, Lemma~\ref{boundedthresholdhelp} gives
\begin{equation*}
    \begin{split}
        V_{\alpha,n}(i,1)-V_{\alpha,n}(i-1,1)&\geq  \frac{1-(1-\alpha)^n}{\alpha}\cdot (h(i-n+1)-h(i-n)) \\
        &=\frac{1-(1-\alpha)^n}{\alpha}\cdot (h(j)-h(j-1)) > \frac{c_{\mu_2}}{(1-\alpha)(\mu_2-\mu_1)}.
    \end{split}
\end{equation*}
By Equation~(\ref{eq:convthresh}) this implies that $\phi_{\alpha,n+1}(i,1)=\mu_2$. It follows that $i_{\alpha,n+1}<i\leq B_n$.

Now, we consider the non-discounted case. Let any 
$j\geq 2$ be given such that $h(j)-h(j-1)>0$. Take $$n=\Big\lceil \frac{c_{\mu_2}}{(h(j)-h(j-1))(\mu_2-\mu_1)}\Big\rceil+1.$$ 
Let $i=j+n-1>n$.
Either $i_n<i+2n:=B_n$, or $i_n\geq B_n$ and
Lemma~\ref{boundedthresholdhelp} gives 
\begin{equation*}
    \begin{split}
        V_{n}(i,1)-V_{n}(i-1,1)&\geq  n(h(j-1)-h(j)) > \frac{c_{\mu_2}}{(\mu_2-\mu_1)}.
    \end{split}
\end{equation*}
By Equation~(\ref{eq:convthresh}) we can conclude that $\phi_{n+1}(i,1)=\mu_2$ and therefore $i_{n+1}<i\leq B_n$.
\end{proof}


Because of the non-increasingness of thresholds $i_{\alpha,n}\in\mathbb{N}_0\cup\{\infty\}$
w.r.t. $n$, the limit $i_\alpha=\lim_{n\rightarrow \infty} i_{\alpha,n} $ is attained within finitely many iterations. Consequently, the non-decreasingness of thresholds $i_{\alpha}\in\mathbb{N}_0\cup\{\infty\}$
w.r.t. $\alpha$ results in strong Blackwell optimality~\cite{Blackwell}.



\begin{definition}\label{BlackDef}
    The policy $\phi^\ast\in \Phi$ is called strongly Blackwell optimal if there exists a value of $\alpha' \in \mathbb{R}_{>0}$, such that $V^{\phi^\ast}_\alpha \leq V^{\phi}_\alpha$ in the component-wise ordering for any $0<\alpha <\alpha'$.
\end{definition}

\begin{proposition}
    Any limiting policy $\phi^\ast$ of VI is strongly Blackwell optimal. 
\end{proposition}

\section{ $\epsilon$-approximation of $s_\alpha(i), \alpha>0$}\label{sect:ConvVIDisc}

We use the convergence of VI for the approximation of the values $s_\alpha(i)$ together with an upper bound $u_M\geq\sup_\phi ||c^\phi||_M$ as derived in Appendix~\ref{App:uM}.
 Combining \cite[Theorem~5.2]{mdppractice}, $V_{\alpha,0}\equiv 0$ 
and the inequality $ || V_{\alpha,1}-V_{\alpha,0} ||_M=|| V_{\alpha,1}||_M \leq  u_M$ from the proof of \cite[Lemma~5.1]{mdppractice}, we obtain  \begin{equation} \label{eq:discountboundje}
    ||V_{\alpha}-V_{\alpha,n}||_M \leq  \frac{(1-\alpha)^n\gamma^n || V_{\alpha,1}-V_{\alpha,0} ||_M}{1-(1-\alpha)\gamma}  \leq  \frac{(1-\alpha)^n\gamma^n u_M}{1-(1-\alpha)\gamma}.
\end{equation} 
Therefore,
\begin{align}
\label{eq:vfan1}
|V_{\alpha}(i,q)-V_{\alpha,n}(i,q)| \leq \frac{(1-\alpha)^n\gamma^n u_M}{1-(1-\alpha)\gamma} \cdot M(i), \quad i\in \mathbb{N}_0, q\in \{0,1\}.
\end{align}
This leads to the following bounds of $V_{\alpha}(i,1),V_\alpha(i,0)$ 
\begin{align}
\label{eq:vfan2}
    0\leq V_{\alpha}(i,1)\leq V_{\alpha}(i,0)
    =|V_{\alpha}(i,0)-V_{\alpha,0}(i,0)| 
    \leq \frac{u_M}{1-(1-\alpha)\gamma} M(i).
\end{align}
Fix $i\in \mathbb{N}_0,\epsilon>0$ and let
\begin{equation}\label{eq:ni}
    n_{\epsilon,i}=
    \Bigg\lceil \log_{(1-\alpha)\gamma}\Bigg(\frac{\epsilon(1-(1-\alpha)\gamma)}{2u_M\cdot \gamma^i}\Bigg) \Bigg\rceil.
\end{equation}
By Equation~(\ref{eq:vfan1}) 
the expected saved cost $s_{\alpha}(i)$ can be $\epsilon$-approximated by \\
$ s_{\alpha,n_{\epsilon,i}}(i)=V_{\alpha,n_{\epsilon,i}}(i,0)-V_{\alpha,n_{\epsilon,i}}(i,1)$ as
\begin{equation}\label{eq:vfanta}
\begin{split}
    | s_{\alpha,n_{\epsilon,i}}(i)-s_{\alpha}(i) |&= \Big|\Big( V_{\alpha,n_{\epsilon,i}}(i,0)-V_{\alpha,n_{\epsilon,i}}(i,1)\Big)-
\Big(V_{\alpha}(i,0)-V_{\alpha}(i,1)\Big)\Big|\\
&\leq\Big| V_{\alpha}(i,0)-V_{\alpha,n_{\epsilon,i}}(i,0)\Big|+
\Big|V_{\alpha}(i,1)-V_{\alpha,n_{\epsilon,i}}(i,1)\Big|\\
&\leq \frac{2(1-\alpha)^{n_{\epsilon,i}}\gamma^{n_{\epsilon,i}} u_M}{1-(1-\alpha)\gamma} \cdot \gamma^i \leq \epsilon.
\end{split}
\end{equation}
 We consider the $n_{\epsilon/2,i}$-stage optimal policy $\phi_{\alpha,n_{\epsilon/2,i}} \in \Phi$.
 Then,
$| s_{\alpha,n_{\epsilon/2,i}}(i)-s_{\alpha}(i) |\leq \frac{\epsilon}{2}$, by virtue of
\cite[Theorem~5.2]{mdppractice} 
the policy $\phi_{\alpha,n_{\epsilon/2,i}}$ is $\epsilon$-optimal for starting state $(i,1)$, i.e.,\\ 
$$|V^{\phi_{\alpha,n_{\epsilon/2,i}}}_\alpha(i,1)-V_\alpha (i,1)|=s^{\phi_{\alpha,n_{\epsilon/2,i}}}_\alpha(i)-s_\alpha(i) \leq \epsilon.$$ We note that $n_{\epsilon/2,i}$ is non-decreasing in $i$, and thus, by Equation~(\ref{eq:vfanta})
,policy $\phi_{\alpha,n_{\epsilon/2,i}}$ is also $\epsilon$-optimal for all starting states $(j,1)$ with $j\leq i$. Furthermore, policy $\phi_{\alpha,n_{\epsilon/2,i}}$ is of a threshold form by Proposition~\ref{increasingconvexvalues}.

By applying Proposition~\ref{increasingconvexvalues} we can determine whether the choice of $\epsilon$ with resulting $n_{\epsilon,i}$ is sufficient to guarantee that the threshold $i_{\alpha,n_{\epsilon,i}}$ is $\alpha$-discount optimal.
\begin{remark}\label{optpoldiscount}
Let $i\in \mathbb{N}_0, \epsilon>0$ be given with $n_{\epsilon,i}$ from Equation~(\ref{eq:ni}). If $i\geq i_{\alpha,n_{\epsilon,i}}+1$ and
\begin{eqnarray*}
V_{\alpha,n_{\epsilon,i}}(
i_{\alpha,n_{\epsilon,i}},1)-V_{\alpha,n_{\epsilon,i}}((
i_{\alpha,n_{\epsilon,i}}-1)^+,1)+2\epsilon &\leq& \frac{c_{\mu_2}}{(1-\alpha)(\mu_2-\mu_1)}\\
&\leq& V_{\alpha,n_{\epsilon,i}}(
i_{\alpha,n_{\epsilon,i}}+1,1)-V_{\alpha,n_{\epsilon,i}}(
i_{\alpha,n_{\epsilon,i}},1)-2\epsilon,    
\end{eqnarray*}
 then the threshold policy $\phi^\ast\in \Phi$ with threshold $i^\ast_\alpha=i_{\alpha,n_{\epsilon,i}}$, defined by $$\phi^\ast(i,1) =\begin{cases}
        \mu_1 \ \ \ \ \ \text{ if } \ \ i\leq i^\ast_\alpha,\\
         \mu_2 \ \ \ \ \ \text{ if } \ \ i> i^\ast_\alpha,
    \end{cases}$$
    is $\alpha$-discount optimal.
\end{remark}


A priori, there is no criterion how to chose $\epsilon$ as the increments can be arbitrarily close to (or even equal to) 
the threshold value. 

For the expected cost difference between the queue using a specific policy $\phi$ and the fixed rate queue, 
we refer to Appendix~\ref{sect:specpol}.

\section{$\epsilon$-approximation of the non-discounted expected saved costs $s(i)$}\label{sect:nondmod}

Clearly, we can restrict VI to the actions of an optimal policy $\phi^\ast$ and VI still converges. 
   \begin{remark} \label{limitsi}
The relative values relate directly to $ s(i)$ through $$s(i):=\mathbb{E}^{\phi^\ast}_{\delta_i^0 \times \delta_i^1} \Big[ \sum_{n=0}^\infty \Big( c(X_n)- c(Y_n)\Big) \Big]=H^\ast(i,0)-H^\ast(i,1) .$$
\end{remark}

Consider the unique $M$-bounded 
solution $(g^\ast,H^\ast)$ of the DAOE~(\ref{DAOE}) with $H^\ast(0,0)=0 $.
In order to approximate $H^\ast$,
we need to bound its
$M$-norm. 
For this matter, we need
the taboo matrix $\prescript{}{(0,0)}P^\phi$, which is the matrix $P^\phi$ with all entries in the column of state $(0,0)$ replaced by zeroes. As $H^\ast(0,0)=0$ we note that $P^\phi H^\ast = \prescript{}{(0,0)}P^\phi H^\ast$. Hence, 

\begin{equation*}
    \begin{split}
        H^\ast=
        \min_{\phi \in \Phi} \Big\{c^\phi-g^\ast \bold{1}+ \prescript{}{(0,0)}P^\phi H^\ast \Big\},  
    \end{split}
\end{equation*}

where $\bold{1}$ is the all-ones vector.
Inserting minimising policy $\phi^\ast$, we find
\begin{equation*}
    \begin{split}\label{eq:Hbound1}
         H^\ast=c^{\phi^\ast}-g^\ast\bold{1}+\prescript{}{(0,0)}P^{\phi^\ast} H^\ast &\implies (I-\prescript{}{(0,0)}P^{\phi^\ast})H^\ast=c^{\phi^\ast}-g^\ast\bold{1}\\
         &\implies H^\ast=\sum_{n=0}^\infty (\prescript{}{(0,0)}P^{\phi^\ast})^n (c^{\phi^\ast}-g^\ast\bold{1}).
    \end{split}
\end{equation*}

Let, w.l.o.g.,
$u_M\geq \max\{ \sup_\phi ||c^\phi||_M,||g^\ast\bold{1}||_M\}$, 
yielding
\begin{equation}
    \begin{split}\label{eq:Hbound2}
        || H^\ast||_M\leq \sum_{n=0}^\infty (\prescript{}{(0,0)}P^{\phi^\ast})^n ||c^{\phi^\ast}-g^\ast\bold{1}||_M\leq   \sum_{n=0}^\infty (\prescript{}{(0,0)}P^{\phi^\ast})^n u_M  \leq u_M \sum_{n=0}^\infty  \eta^n =\frac{u_M}{1-\eta} .
    \end{split}
\end{equation}

Now, we can approximate $H^\ast$ in two different ways. One way is through $M$-uniform geometric recurrence, as we show with Lemma~\ref{HAppr2} in Appendix~\ref{sect:MProofs}. This is an easy method but results in a very high number of iterations, when $\eta$ is close to $1$.

Alternatively, we show we can $\epsilon$-approximate $H^\ast(\cdot,1)$ using a related discounted model with discount factor $\beta>0$, of which $H^\ast$ is the unique value function in a suitably normed space. This method is less affected by $\eta$.
This $\beta$-discounted model is defined by state space $S'=\mathbb{N}_0$, 
action space $\mathcal{A}'=\{\mu_1,\mu_2\}$ for all the states and costs ${c'}^{\phi}(i)=c_a+h(i)-g_\mu +\beta H^\ast(i,0)$. We assume the fixed values of $g_\mu$ and $H^\ast(i,0)$ to be known.
The transition rates from $i\in \mathbb{N}_0$ to $j\in \mathbb{N}_0$, given policy $\phi$ with $\phi(i)=a$ are given by 
$$ {p'}^{\phi}_{i,j} = \begin{cases} \frac{\lambda}{1-\beta} \ \ \ \ \  
&\text{for } j=i+1,\\
\frac{a}{1-\beta} \mathds{1}_{\{i>0\}} \ \ \ \ \ & \text{for } j=i-1,\\
1-\frac{\lambda+a\mathds{1}_{\{i>0\}}}{1-\beta} \ \ \ \ \ &  \text{for } j=i,\\
0 \ \ \ \ \ \ \ \ \ \ &  \text{otherwise. }\end{cases}$$


\begin{restatable}{lemma}{HApproximate}\label{HAppr}
  Let $V'_{\beta,0}\equiv 0$. For any $\epsilon>0,k\in \mathbb{N}_0$,
 we can compute $\epsilon$-approximations \\
$\tilde{H}_\epsilon(i,1)=V'_{\beta,\Tilde{n}_{\epsilon,k}}(i),0\leq i \leq k$ of the values $H^\ast(0,1),H^\ast(1,1),\dots,H^\ast(k,1)$. 
Here, 
\begin{equation}\label{eq:nek}
\begin{split}
\Tilde{n}_{\epsilon,k}=\Bigg\lceil \log_{(1-\beta)\zeta}\Big(\frac{\epsilon(1-(1-\beta)\zeta)}{\tilde{K}\cdot \zeta^k}\Big)\Bigg\rceil
\end{split}\end{equation}
is the number of iterations of VI
with w.l.o.g., $1<\zeta<1/(1-\beta)$ and $$ \tilde{K}= g_\mu+ u_{M} +\frac{\beta \cdot u_M}{1-\eta}.$$
\end{restatable}
\begin{proof}

We will show that solutions of the DDOE of this problem coincide with $H^\ast(\cdot,1)$ and we use an approximation for the discounted problem. 
The corresponding DDOE~(\ref{DDAO}) is defined/given by
\begin{equation}\label{lemmaddoe}
\begin{split}
        u(i) &= \min_{a\in \{1,2\}}\Big\{c_a+h(i)-g_\mu+\beta H^\ast(i,0)\\
        &\ \ \ \ +(1-\beta)(\frac{\mu_a}{1-\beta} \mathds{1}_{\{i>0\}} u(i-1)+(1-\frac{\lambda+\mu_a\mathds{1}_{\{i>0\}}}{1-\beta})u(i)+\frac{\lambda}{1-\beta} u(i+1) ) \Big\}\\
        &= \min_{a\in \{1,2\}}\Big\{c_a+h(i)-g_\mu+\beta H^\ast(i,0)\\
        &\ \ \ \ +(\mu_a \mathds{1}_{\{i>0\}} u(i-1)+(1-\beta-\lambda-\mu_a\mathds{1}_{\{i>0\}})u(i)+\lambda u(i+1) )\Big\}.
\end{split}
\end{equation}


We need to show that Assumption~\ref{assump2alpha} holds such that we can consider the expected total $\beta$-discounted cost value function of this model $V'_\beta$ 
as solution of the DDOE.

Define $\tilde{M}(i)=M(i,1)=\zeta^i$.
As in Section~\ref{sect:ConvVI}, $\tilde{M}$ is a $(\zeta,\Phi)$-drift function. 
Clearly $u_M\geq||c_{\mu_2}+h||_{\tilde{M}}$.
We note 
\begin{eqnarray}\label{eq:eerstebound}
\sup_\phi|| {c'}^\phi||_{\tilde{M}}\leq g_\mu+||c_{\mu_2}+h ||_{\tilde{M}} +
\beta|| H^\ast(\cdot  ,0)||_{\tilde{M} }&\leq& g_\mu+ u_{M} +\beta|| H^\ast(\cdot,0)||_{\tilde{M}}.    
\end{eqnarray}

Combining Equation~(\ref{eq:Hbound2})
with Equation~(\ref{eq:eerstebound}), we get
\begin{equation}\label{eq:tweedebound}
    \sup_\phi|| {c'}^\phi||_{\tilde{M}}\leq 
g_\mu+ u_{M} +\frac{\beta \cdot u_M}{1-\eta}=\tilde{K}.
\end{equation}
Then, we consider VI for this model with $n$-th stage value vectors $V'_{\beta,n}$ for $V'_{\beta,0}\equiv 0$.
We note that \\
$|| V'_{\beta,1}-V'_{\beta,0} ||_{\tilde{M}}\leq  \sup_\phi|| {c'}^\phi||_{\tilde{M}}\leq 
\tilde{K}$.
Assumption~\ref{assump2alpha} implies the equivalent of Equation~(\ref{eq:discountboundje}) for this model to hold 
$$ |V'_\beta(i)-V'_{\beta,n}(i)| <  \frac{(1-\beta)^n \zeta^n \tilde{K}}{1-(1-\beta)\zeta} \cdot \zeta^i.$$ 
Thus, for $\tilde{n}_{\epsilon,k}$ as defined in Equation~(\ref{eq:nek})
such that
$|V'_\beta(i)-V'_{\beta,\Tilde{n}_{\epsilon,k}}(i)| \leq \epsilon$, for all $i=0,1,\dots,k$.

$V'_\beta$ is the unique $\tilde{M}$-bounded solution of the DDOE. 
Comparing the DDOE of Equation~(\ref{lemmaddoe}) with  the DAOE~(\ref{DAOE})
gives the equivalence $V'_\beta=H^\ast(i,1)$.
This gives the equivalence of $\epsilon$-approximations $\tilde{H}_\epsilon(i,1)=V'_{\beta,\Tilde{n}_{\epsilon,k}}(i)$ for $0\leq i \leq k$.
\end{proof}

As a consequence, we can use Lemma~\ref{HAppr} 
with Remark~\ref{limitsi} to determine $\epsilon$-approximations of $s(i),i\geq 0$ given by \begin{equation}\label{eq:sepsiloni}
\Tilde{s}_\epsilon(i):=H^\ast(i,0)-H_\epsilon(i,1).
\end{equation}

Again, we can apply Proposition~\ref{increasingconvexvalues} to determine whether the choice of $\epsilon$, with  resulting $\epsilon$-approximation, is sufficient to guarantee that the threshold $i_{\tilde{n}_{\epsilon,k}}$ is optimal.


\begin{remark}\label{optpolnondiscount}
 Let $\epsilon>0,k\in \mathbb{N}_0$ with $\epsilon$-approximations $\Tilde{H}_\epsilon(0,1),\Tilde{H}_\epsilon(1,1),\dots,\Tilde{H}_\epsilon(k,1)$ be given. 
 If there exists 
 $0\leq i^\ast \leq k-1$
 such that
\begin{equation*}
    \begin{split}
    \tilde{H}_\epsilon(i^\ast,1)- \tilde{H}_\epsilon((i^\ast-1)^+,1)+2\epsilon \leq \frac{c_{\mu_2}}{\mu_2-\mu_1} \leq \tilde{H}_\epsilon(i^\ast+1,1)- \tilde{H}_\epsilon(i^\ast,1)-2\epsilon,
    \end{split}
\end{equation*} 
then there exists optimal threshold policy $\phi^\ast \in \Phi$ with threshold $i^\ast$, defined by $$\phi^\ast(i,1) =\begin{cases}
        \mu_1 \ \ \ \ \ \text{ if } \ \ i\leq i^\ast,\\
         \mu_2 \ \ \ \ \ \text{ if } \ \ i> i^\ast.
    \end{cases}$$
\end{remark}

Again, performance guarantees for finding an optimal policy in this way are not possible as the increments can be arbitrarily close to the threshold value.

\begin{restatable}{proposition}{proponondiscounti}\label{corolnondiscounti}
    Let $i\in \mathbb{N}_0$.
    The $\Tilde{n}_{\epsilon/2,i}$-stage optimal policy $\tilde{\phi}_{\Tilde{n}_{\epsilon/2,i}}$ for the  $\beta$-discounted model defines an $\epsilon$-optimal policy $\phi^i_{\epsilon}\in \Phi$
    for the non-discounted model with initial queue length $i\in \mathbb{N}_0$.

\end{restatable}
\begin{proof}
 
Take $\Tilde{n}_{\epsilon/2,i}$ as in Equation~(\ref{eq:nek}).
Then, 
$|V'_\beta(i)-V'_{\beta,\Tilde{n}_{\epsilon/2,i}}(i)| \leq \frac{\epsilon}{2}$. 
     \cite[Theorem~5.2]{mdppractice} gives that
the $\Tilde{n}_{\epsilon/2,i}$-stage optimal policy $\tilde{\phi}_{\Tilde{n}_{\epsilon/2,i}}$ 
is $\epsilon$-optimal for this $\beta$-discounted model with starting state $i$ (as well as for any state $0\leq j \leq i$). 
For the non-discounted model over $S$, we can fix the policy $\phi_\epsilon^i$ which takes actions $\tilde{\phi}_{\Tilde{n}_{\epsilon/2,i}}(j)$ in states $(j,1)\in S$. Denote the $M$-bounded relative value function of the resulting model by $H^{\phi_\epsilon^i}$.  
Noting the equivalence between ${V'}^{\tilde{\phi}_{\Tilde{n}_{\epsilon/2,i}}}_\beta$ and $H^{\phi_{\epsilon}^i}(\cdot,1)$ through the proof of Lemma~\ref{HAppr} gives 
\begin{equation*}
    \begin{split}
       s(i)\geq\mathbb{E}^{\phi^i_{\epsilon}}_{\delta_i^0\times\delta_i^1} \Bigg[ \sum_{n=0}^\infty \Big( c(X_n)- c(Y_n) \Big) \Bigg] &=  H^{\phi_{\epsilon}^i}(i,0)-H^{\phi_{\epsilon}^i}(i,1)= H^\ast(i,0)-{V'}^{\tilde{\phi}_{\Tilde{n}_{\epsilon/2,i}}}_\beta(i) \\
       &\geq H^\ast(i,0)-V'_\beta(i)-\epsilon =H^\ast(i,0)-H^\ast(i,1)-\epsilon =s(i)-\epsilon.
    \end{split}
\end{equation*}

\end{proof}

\section{The expected value saved when starting in the stationary distribution}\label{sect:SavedVals}

For specific starting states, we have 
determined $\epsilon$-approximations of $s_{\alpha}(i),i\geq 0$ (cf. Equation~(\ref{eq:vfanta})) and of $s(i),i\geq 0$ (cf. Equation~(\ref{eq:sepsiloni})).
Using these approximations, we will compute $\epsilon$-approximations of $s_\alpha(\pi_\mu)$ and $s(\pi_\mu)$  in Sections~\ref{sect:discpi} and \ref{sect:nondiscpi}, respectively.
 Furthermore, we will determine $\epsilon$-optimal 
 policies for both cases when the initial queue length distribution is $\pi_\mu$. 





If an unobservable M/M/1 queue with fixed rate $\mu$ is stable, the stationary distribution of the queue length is given by $\pi_{\mu}$. One might wonder what the amount of expected saved cost is by introducing a single period of control (including observation). This is precisely $s_\alpha(\pi_\mu),\alpha\in[0,1)$. 

For probability distributions $\pi$ and function $f$ on a state space $S$, we
introduce the notation of the expected value of $f$ w.r.t. $\pi$ as $f(\pi)=\sum_{y\in S}\pi(y)f(y)$. 

\subsection{The discounted saved value $s_{\alpha}(\pi_{\mu})$}\label{sect:discpi}


To determine an $\epsilon$-approximation of $s_\alpha(\pi_\mu)$ we use the $\epsilon$-approximations derived in Section~\ref{sect:ConvVIDisc}.
We restrict the choice of $\gamma$ 
such that $\gamma \rho_{\mu}<1$. This implies that $V_\alpha$ is integrable w.r.t. $\pi_\mu$.

\begin{restatable}{theorem}{discountfind}\label{DiscountFind}
   For $\epsilon>0$, we can $\epsilon$-approximate the expected saved cost $s_{\alpha}(\pi_{\mu})$ by  $$\Tilde{s}_{\alpha,\epsilon}(\pi_{\mu})=\sum_{i=0}^k \pi_{\mu}(i) V_{\alpha,n_{\epsilon/2,k}}(i,0)-\sum_{i=0}^k \pi_{\mu}(i)V_{\alpha,n_{\epsilon/2,k}}(i,1), $$ 
   where 
   \begin{equation}\label{eq:kform}
k=\Bigg\lceil\log_{\gamma \rho_{\mu}}\Big(\frac{\epsilon(1-\gamma \rho_{\mu})(1-(1-\alpha)\gamma ) }{2u_M(1-\rho_{\mu})}\Big)  \Bigg\rceil .
\end{equation}

   and $n_{\epsilon/2,k}$ as defined in Equation~(\ref{eq:ni}).
\end{restatable}
\begin{proof}
    The expected saved cost $s_{\alpha}(\pi_{\mu})$ is given by $ \sum_{i=0}^\infty \pi_{\mu}(i)V_{\alpha}(i,0)-\sum_{i=0}^\infty \pi_{\mu}(i)V_{\alpha}(i,1)$. Furthermore,
    $$|V_{\alpha}(i,0)-V_{\alpha,n_{\epsilon/2,k}}(i,0)|,|V_{\alpha}(i,1)-V_{\alpha,n_{\epsilon/2,k}}(i,1) | \leq \epsilon/4, \quad i\leq k, $$ 
    cf. Equation~(\ref{eq:vfanta}). 
    As $V_{\alpha,n_{\epsilon/2,k}}(i,0)\geq V_{\alpha,n_{\epsilon/2,k}}(i,1)$, we have 
    $$0\leq \sum_{i=k+1}^\infty \pi_{\mu}(i) V_{\alpha,n_{\epsilon/2,k}}(i,0)-\sum_{i=k+1}^\infty  \pi_{\mu}(i)V_{\alpha,n_{\epsilon/2,k}}(i,1)\leq \sum_{i=k+1}^\infty \pi_{\mu}(i) V_{\alpha,n_{\epsilon/2,k}}(i,0).$$  We proceed by using Equation (\ref{eq:vfan2}),
    and find
\begin{align*}
    \begin{split}
        |s_{\alpha}(\pi_{\mu})-\Tilde{s}_{\alpha,\epsilon} (\pi_{\mu})|
 &\leq \sum_{i=0}^k \pi_{\mu}(i)(|V_{\alpha}(i,0)-V_{\alpha,n_{\epsilon/2,k}}(i,0)|+|V_{\alpha}(i,1)-V_{\alpha,n_{\epsilon/2,k}}(i,1) |) \\
 &\quad+\sum_{i=k+1}^\infty \pi_{\mu}(i)V_{\alpha}(i,0)\\ 
 &\leq \frac{\epsilon( 1-\rho_{\mu})}{2}\sum_{i=0}^k \rho_{\mu}^i +(1-\rho_{\mu})\sum_{i=k+1}^\infty \rho_{\mu}^i\Big( \frac{u_M}{1-(1-\alpha)\gamma} \gamma^i\Big) \\
 &= \frac{\epsilon(1-\rho_{\mu}^{k+1})}{2}+ \frac{u_M(1-\rho_{\mu})}{1-(1-\alpha)\gamma}\cdot\sum_{i=k+1}^\infty (\rho_{\mu} \gamma)^i \\
&<\frac{\epsilon}{2}+ \frac{u_M(1-\rho_{\mu})}{1-(1-\alpha)\gamma}\cdot \frac{(\gamma \rho_{\mu})^{k+1}}{1-\gamma \rho_{\mu}}.
 \end{split}
\end{align*}
Thus for the chosen value of $k$ we find 
\begin{align*}
    \begin{split}
        |s_{\alpha}(\pi_{\mu})-\Tilde{s}_{\alpha,\epsilon} (\pi_{\mu})|
&<\frac{\epsilon}{2}+ \frac{u_M(1-\rho_{\mu})}{1-(1-\alpha)\gamma}\cdot \frac{(\gamma \rho_{\mu})^{k+1}}{1-\gamma \rho_{\mu}} \leq \frac{\epsilon}{2}+\frac{\epsilon}{2}=\epsilon.
 \end{split}
\end{align*}

\end{proof}

Note that this approximation requires the computation of value functions $V_{\alpha,m}$, for which a finite truncation suffices.



If we do not manage to find the optimal policy $\phi^\ast$ using Remark~\ref{optpoldiscount}, we are still able to determine an $\epsilon$-optimal policy.

\begin{restatable}{proposition}{coroldiscount}\label{coroldiscount}
For any $\epsilon>0$, the $n_{\epsilon/2,k}$-stage optimal policy $\phi_{n_{\epsilon/2,k}}\in \Phi$ is
$\epsilon$-optimal
for initial queue length distribution $\pi_\mu$.
\end{restatable}
\begin{proof}
   We consider the expected total discounted value function $V^{\phi_{n_{\epsilon/2,k}}}_\alpha $ for $n_{\epsilon/2,k}$-stage optimal policy $\phi_{n_{\epsilon/2,k}}\in \Phi$.
    \cite[Theorem~5.2]{mdppractice} gives that 
    \begin{equation*}
        \begin{split}
             |V_{\alpha}(i,1)-V_{\alpha}^{\phi_{n_{\epsilon/2,k}}}(i,1) | &\leq \epsilon/2  \hspace{1.2cm}\text{   for  }i\leq k,\\
             |V_{\alpha}(i,1)-V_{\alpha}^{\phi_{n_{\epsilon/2,k}}}(i,1) | &\leq \gamma^{i-k}\epsilon/2   \hspace{0.45cm}\text{   for  }i>k.
        \end{split}
    \end{equation*}
    W.l.o.g. $u_M\geq 1-(1-\alpha)\gamma$. 
    Conclusively, the proof of Theorem~\ref{DiscountFind} applies to 
    $V^{\phi_{n_{\epsilon/2,k}}}_\alpha $ as well, giving 
  \begin{equation*}
\begin{split}
    &\Big|\sum_{i=0}^\infty \pi_{\mu}(i)V_\alpha(i,0)-\sum_{i=0}^\infty \pi_{\mu}(i)V^{\phi_{n_{\epsilon/2,k}}}_\alpha(i,1) -s_\alpha(\pi_{\mu}) \Big| =  \Big|\sum_{i=0}^\infty \pi_{\mu}(i)V_\alpha(i,1)-\sum_{i=0}^\infty \pi_{\mu}(i)V^{\phi_{n_{\epsilon/2,k}}}_\alpha(i,1) \Big|\\
   &\quad\quad\quad\quad \quad\quad\quad\quad \quad\quad \leq  \sum_{i=0}^k \pi_{\mu}(i)\Big|V_\alpha(i,1)-V^{\phi_{n_{\epsilon/2,k}}}_\alpha(i,1) \Big| + \sum_{i=k+1}^\infty \pi_{\mu}(i)\Big|V_\alpha(i,1)-V^{\phi_{n_{\epsilon/2,k}}}_\alpha(i,1) \Big| \\
   &\quad\quad\quad\quad\quad\quad\quad\quad \quad\quad < \frac{\epsilon}{2}+ (1-\rho_{\mu})\sum_{i=k+1}^\infty (\gamma \rho_{\mu})^i \leq \frac{\epsilon}{2}+ \frac{u_M(1-\rho_{\mu})}{1-(1-\alpha)\gamma}\cdot \frac{(\gamma \rho_{\mu})^{k+1}}{1-\gamma \rho_{\mu}}  <\epsilon.
\end{split}      
  \end{equation*}

\end{proof}

Theorem~\ref{DiscountFind} 
gives a method to approximate the expected saved cost for any discount rate.
Proposition~\ref{coroldiscount} then enables to determine an approximating policy.

The needed number of calculations explodes when $\alpha\downarrow 0$ (due to multiple influences through $\alpha$). Therefore, this theorem is not be suitable for combination with a vanishing discount approach ($\alpha\downarrow 0$) to approximate the expected saved cost for the non-discounted model.

\subsection{The non-discounted model}\label{sect:nondiscpi}

In the following section, we will approximate the expected saved cost $s(\pi_{\mu})= H(\pi_{\mu}^0)-H(\pi_{\mu}^1)$.

\begin{restatable}{theorem}{nondiscountfind}\label{NonDiscountFind}
 For $\epsilon>0$, an $\epsilon$-approximation of the expected saved cost $s(\pi_{\mu})$ is given by  $$\Tilde{s}_{\epsilon}(\pi_{\mu})= \sum_{i=0}^{\tilde{k}_\epsilon} (1-\rho_\mu)\rho_\mu^i (H^\ast(i,0)-\Tilde{H}_{\epsilon/2}(i,1)), $$ 
   where 
\begin{equation}\label{eq:kform2}
\tilde{k}_\epsilon= 
\Bigg\lceil\log_{\rho_\mu \zeta} \Bigg(\frac{\epsilon(1-\eta)(1-\rho_\mu \zeta)}{4u_M (1-\rho_\mu)}\Bigg) \Bigg\rceil,
\end{equation}
 and $\Tilde{H}_{\epsilon/2}(i,1),i=0,1,\hdots,\tilde{k}_\epsilon$ are the $\epsilon/2$-approximations as  determined in Lemma~\ref{HAppr}.
\end{restatable}
\begin{proof}
Note that $\rho_\mu \zeta <1$ (as $\zeta< \mu/\lambda$) such that $\tilde{k}_\epsilon$ is well-defined.  
 By Equation~(\ref{eq:Hbound2})
\begin{equation}\label{eq:HpiTail}
    \begin{split}
 \hspace*{-0.3cm}\sum_{i=\tilde{k}_\epsilon+1}^\infty \!\!(1-\rho_{\mu})\rho_{\mu}^i |H^\ast(i,0)-H^\ast(i,1)|   
&\leq  \sum_{i=\tilde{k}_\epsilon+1}^\infty \!\! (1-\rho_{\mu})\rho_{\mu}^i \cdot \frac{2u_M}{1-\eta} \zeta^i =\frac{2u_M (1-\rho_{\mu})}{1-\eta} \!\sum_{i=\tilde{k}_\epsilon+1}^\infty \!\!( \rho_\mu \zeta)^i \\
&=  \frac{2u_M (1-\rho_{\mu})(\rho_\mu \zeta)^{\tilde{k}_\epsilon+1}}{(1-\eta)(1-\rho_\mu \zeta)} \leq \frac{\epsilon}{2}.
    \end{split}
\end{equation}

Equation~(\ref{eq:HpiTail}) gives that
\begin{equation*}
    \begin{split}
        |s(\pi_{\mu})-\Tilde{s}_{\epsilon}(\pi_{\mu})|&\leq \sum_{i=\tilde{k}_\epsilon+1}^\infty \!\!(1-\rho_{\mu})\rho_{\mu}^i |H^\ast(i,0)-H^\ast(i,1)|+ \sum_{i=0}^{\tilde{k}_\epsilon} (1-\rho_{\mu})\rho_{\mu}^i |\Tilde{H}_{\epsilon/2}(i,1)-H^\ast(i,1)| \\
        &\leq \frac{\epsilon}{2}+ \frac{\epsilon}{2}\cdot \sum_{i=0}^{\tilde{k}_\epsilon} (1-\rho_{\mu})\rho_{\mu}^i \leq \epsilon.
    \end{split}
\end{equation*}

\end{proof}

In order to determine an $\epsilon$-optimal policy, we define the policy $\phi_0\in \Phi$ to be the policy that only specifies to use service rate $\mu$, i.e.,
$\phi_0(i,q)=\mu $ for $(i,q)\in S$. 


\begin{restatable}{proposition}{corolnondiscount}\label{corolnondiscount}
Let $\epsilon>0$ and let $\tilde{k}$ be defined as in Equation~(\ref{eq:kform2}). Let $\phi_\epsilon$ to be the policy that uses policy $\phi^{\tilde{k}}_{\epsilon/2}$ from Proposition~\ref{corolnondiscounti} for initial queue lengths $i\leq \tilde{k}$ and uses policy $\phi_0$ for initial queue lengths $i > \tilde{k}$.
Policy $\phi_\epsilon$
is $\epsilon$-optimal for initial queue length distribution $\pi_\mu$.
\end{restatable}
\begin{proof}
 
By virtue of the proof of Proposition~\ref{corolnondiscounti} (noting that the choice of $\tilde{n}_{\epsilon/4,\tilde{k}}$ is also sufficient for all values $0\leq i<\tilde{k}$) 
$$\mathbb{E}^{\phi^{\tilde{k}}_{\epsilon/2}}_{\delta_i^0\times\delta_i^1} \Bigg[ \sum_{n=0}^\infty \Big(c(X_n)- c(Y_n)\Big) \Bigg]
\geq s(i)-\frac{\epsilon}{2}, \quad \quad \text{ for }i\leq \tilde{k}.$$
     Note that
$\sum_{i=\tilde{k}+1}^\infty\pi_\mu(i)s(i)\leq \frac{\epsilon}{2} $. Then,

  \begin{equation*}
      \begin{split}
          \mathbb{E}^{\phi_\epsilon}_{\pi_\mu^0\times \pi_\mu^1} \Bigg[ \sum_{n=0}^\infty \Big(c(X_n)- c(Y_n) \Big)\Bigg]&=\sum_{i=0}^{\tilde{k}}\pi_\mu(i) \mathbb{E}^{\phi^{\tilde{k}}_{\epsilon/2}}_{\delta_i^0\times\delta_i^1} \Bigg[ \sum_{n=0}^\infty \Big(c(X_n)- c(Y_n)\Big) \Bigg]\\
          &\geq \sum_{i=0}^{\tilde{k}}\pi_\mu(i)(s(i)-\epsilon/2)\geq \sum_{i=0}^{\tilde{k}}\pi_\mu(i)s(i) -\frac{\epsilon}{2}\\
          &\geq \sum_{i=0}^\infty \pi_\mu(i)s(i)-\frac{\epsilon}{2}-\frac{\epsilon}{2}=s(\pi_\mu)-\epsilon.
      \end{split}
  \end{equation*}
\end{proof}

We need to remark that this $\epsilon$-optimal policy is strictly speaking not a policy, 
as the chosen policy is decided at the initialisation.


\section{Computational results}\label{sect:num}

In this section, we apply the methods discussed in this paper for the discounted and the non-discounted model 
for both linear and quadratic holding costs. 
Visualisations and results for the discounted model are given in Section~\ref{subsectdisc} and for the non-discounted model in Section~\ref{subsectnondisc}, whereas supplementary  results are given in Appendix~\ref{sect:ExtraPlots}.

We consider linear 
holding cost function $h_{lin}(i)=K_{lin}\cdot i$, and quadratic holding cost function \\
$h_{sqd}(i)=K_{sqd}\cdot i^2$. We proceed by taking a sufficient value of $u_M$ from Equations~(\ref{eq:umpol}) and (\ref{eq:umpol2}) in Appendix~\ref{App:uM} for the discounted case and non-discounted case, respectively. 
Next, we calculate $g$ explicitly as $g=\sum_{i=0}^\infty h(i)\cdot \pi_\mu(i)$ and take the relative values $H^\ast(i,0)=H_\mu(i)$ as given by Equations~(\ref{hmulin}) and (\ref{hmusqd}) in Appendix~\ref{sect:HfVer}. In summary, this gives 
\begin{equation*}
\begin{split}
    &u_M=2\Big(c_{\mu_2}+\frac{K_{lin}}{\log(\zeta)\cdot e}\Big),\\
    &g=c_\mu+\sum_{i=0}^\infty K_{lin}i\cdot\pi_{\mu}(i)=c_\mu+K_{lin}(1-\rho_{\mu})\sum_{i=0}^\infty i\cdot\rho_{\mu}^i=c_\mu+\frac{K_{lin} \rho_{\mu}}{1-\rho_{\mu}},\\
    &H^\ast(i,0)=\frac{K_{lin}i(i+1)}{2(\mu-\lambda)} ,
    \end{split}
\end{equation*}
for the linear holding cost function. For quadratic holding costs this results in
\begin{equation*}
\begin{split}
    &u_M=2\Big(c_{\mu_2}+\frac{4K_{sqd}}{\log^2(\zeta)\cdot e^2}\Big),\\
    &g=c_\mu+\sum_{i=0}^\infty K_{sqd}i^2\cdot\pi_{\mu}(i)= c_\mu+K_{sqd}(1-\rho_{\mu})\sum_{i=0}^\infty i^2\cdot\rho_{\mu}^i=c_\mu+\frac{ K_{sqd}\rho_{\mu}(1+\rho_{\mu})}{(1-\rho_{\mu})^2},\\
    &H^\ast(i,0)=\frac{K_{sqd}\cdot i(i+1)( \mu+5\lambda+2i(\mu-\lambda) )}{6(\mu-\lambda)^2} .
    \end{split}
\end{equation*}

 When $H^\ast(i,0),i\geq 0$, can not be determined explicitly, these values can be calculated iteratively using Equation~\ref{Hiter}. However, the numerical iterative computation of these values is unstable with respect to (rounding) errors. 
 Suitable solutions are to use approximations as determined in Appendix~\ref{app:Mrec} or to use Olver's method~\cite{Olver} to numerically solve the second-order difference equation.

\subsection{The discounted model}\label{subsectdisc}

First, we consider the fixed rate $\mu=\mu_1>\lambda$.
Figure~\ref{fig:plot11} contains a plot of the expected saved costs
for one specific combination of parameters.

\begin{figure}[H]
\centering
\begin{subfigure}{.5\textwidth}
  \centering
  \includegraphics[width=1\linewidth]{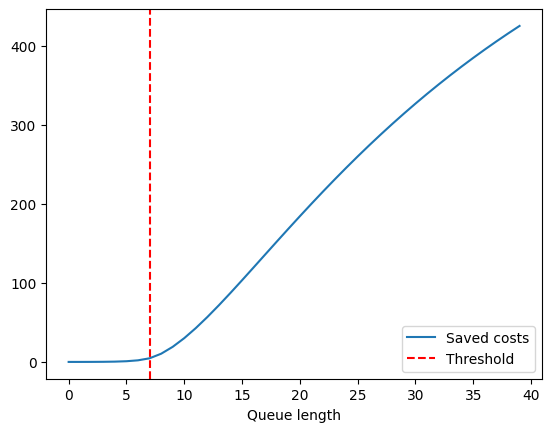}
\end{subfigure}%
\begin{subfigure}{.5\textwidth}
  \centering
  \includegraphics[width=1\linewidth]{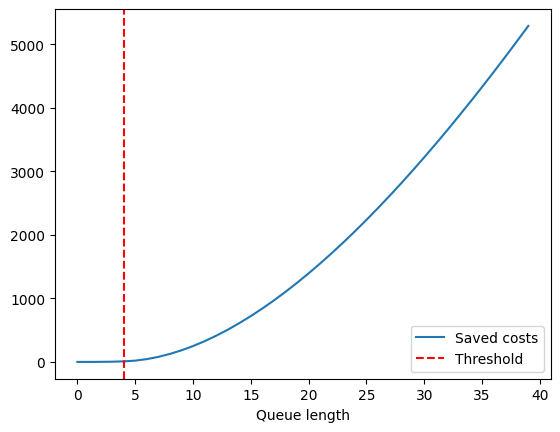}
\end{subfigure}
\caption{ Plots of (approximated) expected saved costs ($\tilde{s}_{\alpha,\epsilon}(i)$) for  $ \alpha=0.005,\beta=0.05, \lambda =0.2,\mu_1=0.35,\mu_2=0.4,c_{\mu_2}=10$, $\mu=\mu_1$ and $\epsilon=0.001$. The two cases correspond to linear holding cost function $h(i)=5\cdot i$, and quadratic holding cost function $h(i)=i^2$, respectively.}
\label{fig:plot11}
\end{figure}

We can summarise the results for several cases with linear holding costs in Table~\ref{table1}, and for cases with for quadratic holding costs in Table~\ref{table2} . We display the expected saved costs starting in distribution $\pi_{\mu}$ and the threshold of the found approximating policies. For each of these, Remark~\ref{optpoldiscount} validates that the found approximating policy is optimal.


\begin{table}[H]
\begin{center}
\caption{Approximations of the expected saved costs $s_{\alpha}(\pi_{\mu})$, thresholds  of optimal policies, values of $k$ and number of iterations for various instances with linear holding costs, $\mu=\mu_1, K_{lin}=5$ and $ \epsilon=0.001$. }
\label{table1}
\begin{tabular}{ |p{0.8cm}|p{0.8cm}|p{0.8cm}|p{0.8cm}|p{0.8cm}|p{0.8cm}|p{1.4cm}|p{1.4cm}|p{0.8cm}|p{1.3cm}|  }
 \hline
 \multicolumn{6}{|c|}{Parameters and costs} & \multicolumn{4}{|c|}{ Results}\\
 \hline
 $\alpha$&  $\beta$ & $\lambda$ & $\mu_1$ & $\mu_2$ & $c_{\mu_2}$ & $\tilde{s}_{\alpha,\epsilon}(\pi_{\mu})$ & Threshold & $k$ & Iterations \\
 \hline
 0.01   & 0.1    & 0.1&   0.35 & 0.45 & 10 &   0.0035 & 5 & 17 & 2216 \\
  0.005 &   0.1    & 0.1&   0.35 & 0.45 & 10 &  0.0073 & 5 & 18 & 4731\\
 0.01 &   0.05    & 0.2&   0.35 & 0.4 & 10 & 0.1316 & 8 & 37 & 2217\\
 0.005   &   0.05    & 0.2&   0.35 & 0.4 & 10 &  0.4642 & 7 & 40 & 4732\\
 0.01&   0.02    & 0.31&   0.33 & 0.34 & 4 &  10.711 & 11 & 333& 2229
\\
 0.005&   0.02   & 0.31&   0.33 & 0.34 & 4 & 99.825 & 5 & 355& 4745
\\
 \hline
\end{tabular}
\end{center}
\end{table}

\begin{table}[H]
\begin{center}
\caption{ Approximations of the expected saved costs $s_{\alpha}(\pi_{\mu})$, thresholds of optimal policies, values of $k$ and number of iterations for various instances with quadratic holding costs, $\mu=\mu_1, K_{sqd}=1 $ and $ \epsilon=0.001$.}
\label{table2}
\begin{tabular}{ |p{0.8cm}|p{0.8cm}|p{0.8cm}|p{0.8cm}|p{0.8cm}|p{0.8cm}|p{1.4cm}|p{1.4cm}|p{0.8cm}|p{1.3cm}|  }
 \hline
 \multicolumn{6}{|c|}{Parameters and costs} & \multicolumn{4}{|c|}{ Results}\\
 \hline
 $\alpha$&  $\beta$ & $\lambda$ & $\mu_1$ & $\mu_2$ & $c_{\mu_2}$ &  $\tilde{s}_{\alpha,\epsilon}(\pi_{\mu})$ & Threshold & $k$ & Iterations \\
 \hline
 0.01   & 0.1    & 0.1&   0.35 & 0.45 & 10 &  0.0269 & 4 & 
22& 2899\\
  0.005 &   0.1    & 0.1&   0.35 & 0.45 & 10 & 0.0383 & 4 & 24 & 6246
\\
 0.01 &   0.05    & 0.2&   0.35 & 0.4 & 10 &  4.2141 & 5 & 49& 2901
\\
 0.005   &   0.05    & 0.2&   0.35 & 0.4 & 10 & 6.3540 & 4 & 
53& 6247
\\
 0.01&   0.02    & 0.31&   0.33 & 0.34 & 10 &  642.084 & 6 & 
439& 2917
\\
 0.005&   0.02   & 0.31&   0.33 & 0.34 & 10 & 1674.233 & 4 & 
472& 6264\\
  0.01&   0.02    & 0.31&   0.33 & 0.34 & 4 &  784.338 & 2 &
439& 2917
\\
 0.005&   0.02   & 0.31&   0.33 & 0.34 & 4 & 1866.601 & 1 & 
472& 6264
\\
 \hline
\end{tabular}
\end{center}
\end{table}

For the discounted model, when $\mu_1 \approx \mu_2$, and when $c_{\mu_2}$ is not low enough, the service rate $\mu_1$ may never be used during control. For this matter, we reduced $c_{\mu_2}$ for the last combination of parameters. We observe that the threshold reduces as the discount factor is halved, as was given in Proposition~\ref{decreasingthreshold}. Furthermore, expected saved costs are more than doubled. This can be explained as the expected saved cost from reduced holding costs is affected more by the discount factor than the additional cost resulting from using the service rate $\mu_2$ (instead of $\mu=\mu_1$) during the (early) control period. Furthermore, the expected saved cost seems to be much higher when $\rho_{\mu}$ is closer to 1. This can be explained by the higher weight of $s_\alpha(i)$ in $s_\alpha(\pi_{\mu})$ for high values of $i$.   

For the fixed rate $\mu=\mu_2$, we plot the expected saved costs of the model for the same specific combination of parameters as before in Figure~\ref{fig:plot12}.

\begin{figure}[H]
\centering
\begin{subfigure}{.5\textwidth}
  \centering
  \includegraphics[width=1\linewidth]{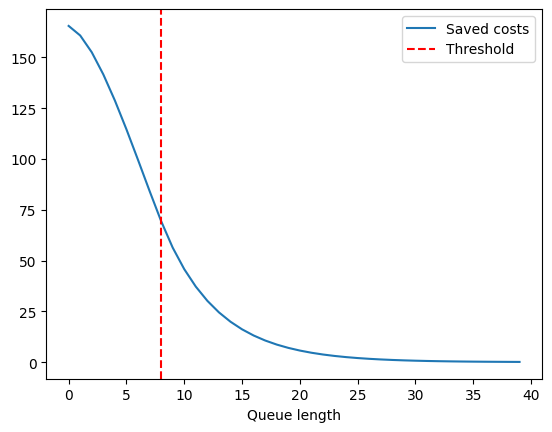}
\end{subfigure}%
\begin{subfigure}{.5\textwidth}
  \centering
  \includegraphics[width=1\linewidth]{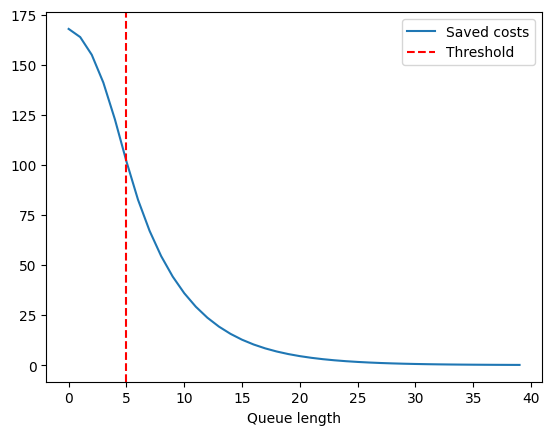}
\end{subfigure}
\caption{ Plots of (approximated) expected saved costs ($\tilde{s}_{\alpha,\epsilon}(i)$) for $\alpha=0.005, \beta=0.05, \lambda =0.2,\mu_1=0.35,\mu_2=0.4,c_{\mu_2}=10$, $\mu=\mu_2$ and $\epsilon=0.001$. The two cases correspond to linear holding cost function $h(i)=5\cdot i$, and quadratic holding cost function $h(i)=i^2$, respectively.}
\label{fig:plot12}
\end{figure}

We summarise the results for several cases in Table~\ref{table3} for linear holding costs, and Table~\ref{table4} for quadratic holding costs. For each of these cases, Remark~\ref{optpoldiscount} validates that the found approximating policy is optimal.

\begin{table}[H]
\begin{center}
\caption{Approximations of the expected saved costs $s_{\alpha}(\pi_{\mu})$, thresholds of optimal policies, values of $k$ and number of iterations for various instances with  linear holding costs, $\mu=\mu_2, K_{lin}=5$ and $  \epsilon=0.001$.}
\label{table3}
\begin{tabular}{ |p{0.8cm}|p{0.8cm}|p{0.8cm}|p{0.8cm}|p{0.8cm}|p{0.8cm}|p{1.4cm}|p{1.4cm}|p{0.8cm}|p{1.3cm}|  }
 \hline
 \multicolumn{6}{|c|}{Parameters and costs} & \multicolumn{4}{|c|}{ Results}\\
 \hline
 $\alpha$&  $\beta$ & $\lambda$ & $\mu_1$ & $\mu_2$ & $c_{\mu_2}$ & $\tilde{s}_{\alpha,\epsilon}(\pi_{\mu})$ & Threshold & $k$ & Iterations \\
 \hline
 0.01   & 0.1    & 0.1&   0.35 & 0.45 & 10 &   87.313 & 7 & 
14& 2216
\\
  0.005 &   0.1    & 0.1&   0.35 & 0.45 & 10 & 90.952 & 6 & 
15& 4731\\
 0.01 &   0.05    & 0.2&   0.35 & 0.4 & 10 & 146.977 & 10 & 
30& 2217\\
 0.005   &   0.05    & 0.2&   0.35 & 0.4 & 10 &  158.037 & 8 & 
32& 4731
\\
 0.01&   0.01    & 0.33&   0.3 & 0.36 & 10 &  66.858 & 2 & 
239& 2225\\
 0.005&   0.01   & 0.33&   0.3 & 0.36 & 10 &  75.489 & 1 & 
255& 4741
\\
 \hline
\end{tabular}
\end{center}
\end{table}

\begin{table}[H]
\begin{center}
\caption{Approximations of the expected saved costs $s_{\alpha}(\pi_{\mu})$, thresholds of optimal policies, values of $k$ and number of iterations for various instances with quadratic holding costs, $\mu=\mu_2,  K_{sqd}=1 $ and $ \epsilon=0.001$.}
\label{table4}
\begin{tabular}{ |p{0.8cm}|p{0.8cm}|p{0.8cm}|p{0.8cm}|p{0.8cm}|p{0.8cm}|p{1.4cm}|p{1.4cm}|p{0.8cm}|p{1.3cm}|  }
 \hline
 \multicolumn{6}{|c|}{Parameters and costs} & \multicolumn{4}{|c|}{ Results}\\
 \hline
 $\alpha$&  $\beta$ & $\lambda$ & $\mu_1$ & $\mu_2$ & $c_{\mu_2}$ &  $\tilde{s}_{\alpha,\epsilon}(\pi_{\mu})$ & Threshold & $k$ & Iterations \\
 \hline
 0.01   & 0.1    & 0.1&   0.35 & 0.45 & 10 &   89.780 & 5& 
19  & 2899
 \\
  0.005 &   0.1    & 0.1&   0.35 & 0.45 & 10 &  93.581 & 5 & 
20& 6245
\\
 0.01 &   0.05    & 0.2&   0.35 & 0.4 & 10 &  148.285 & 5 & 
40& 2900\\
 0.005   &   0.05    & 0.2&   0.35 & 0.4 & 10 & 159.545 & 5 & 
43& 6246
\\
 0.01&   0.01    & 0.33&   0.3 & 0.36 & 10 & 49.377 & 1 & 
315& 2912
\\
 0.005&   0.01   & 0.33&   0.3 & 0.36 & 10 &  55.741 & 0 &
339& 6259
\\
 \hline
\end{tabular}
\end{center}
\end{table}

 Again, we observe Proposition~\ref{decreasingthreshold} in practice.
 When the discount factor is halved, the expected saved costs have not increased by such a substantial fraction. This can be explained as additional saved cost from reduced holding costs is affected more by the discount factor than the reduced cost resulting from using the service rate $\mu_1$ (instead of $\mu=\mu_2$) during the (early) control period. We have included one combination of parameters $ \mu=\mu_2>\lambda \geq \mu_1$. Even though the threshold of the found optimal threshold policy is very low, the expected saved cost can still be substantial.

For figures with graphs of approximations of the expected saved costs $s_\alpha(i)$ for each of the remaining cases in Tables~\ref{table1}, \ref{table2}, \ref{table3} and \ref{table4} with $\alpha=0.005$ we refer to Appendix~\ref{sect:ExtraPlots1}.

\subsection{The non-discounted model}\label{subsectnondisc}

First, we consider the fixed rate $\mu=\mu_1>\lambda$.
For one specific combination of parameters, we plot the expected saved costs in Figure~\ref{fig:plot21}.

\begin{figure}[H]
\centering
\begin{subfigure}{.5\textwidth}
  \centering
  \includegraphics[width=1\linewidth]{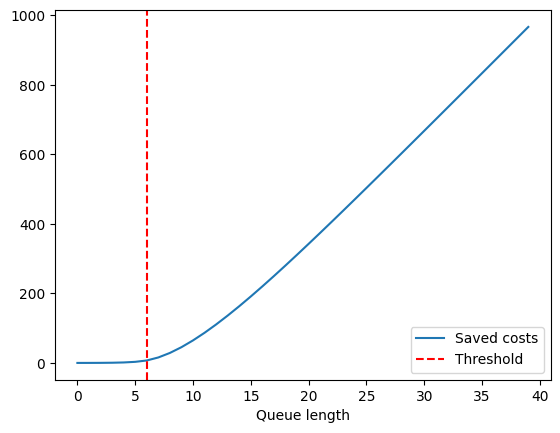}
\end{subfigure}%
\begin{subfigure}{.5\textwidth}
  \centering
  \includegraphics[width=1\linewidth]{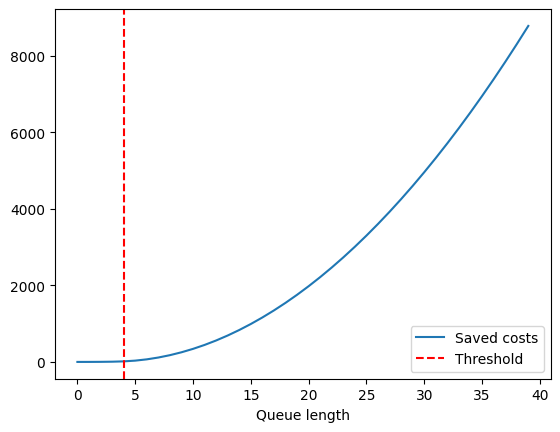}
\end{subfigure}
    \caption{ Plots of (approximated) expected saved costs ($\tilde{s}_\epsilon(i)$) for $\beta=0.05, \lambda =0.2,\mu_1=0.35,\mu_2=0.4,c_{\mu_2}=10, \mu=\mu_1$ and $\epsilon=0.001$. The two cases correspond to linear holding cost function $h(i)=5\cdot i$, and quadratic holding cost function $h(i)=i^2$, respectively.}
\label{fig:plot21}
\end{figure}

We now summarise the results for the linear holding cost cases in Table~\ref{table5} and the quadratic holding cost cases in Table~\ref{table6}. We now display the expected saved costs $s(\pi)$, the threshold of the found approximating policies, the values of $k$ and the number of iterations. Now it is Remark~\ref{optpolnondiscount} that validates that the found approximating policies are optimal. 

\begin{table}[H]
\begin{center}
\caption{Approximations of the expected saved costs $s(\pi_{\mu})$, thresholds of optimal policies, values of $k$ and number of iterations for various instances with linear holding costs, $\mu=\mu_1,  K_{lin}=5 $ and $ \epsilon=0.001$.}
\label{table5}
\begin{tabular}{ |p{0.8cm}|p{0.8cm}|p{0.8cm}|p{0.8cm}|p{0.8cm}|p{1.5cm}|p{1.5cm}|p{0.8cm}|p{1.4cm}|  }
 \hline
 \multicolumn{5}{|c|}{Parameters and costs} & \multicolumn{4}{|c|}{ Results}\\
 \hline
  $\beta$ & $\lambda$ & $\mu_1$ & $\mu_2$ & $c_{\mu_2}$ & $\tilde{s}_\epsilon(\pi_{\mu})$ & Threshold & $k$ & Iterations \\
 \hline
  0.1    & 0.1&   0.35 & 0.45 & 10 &  0.0119 & 5 & 
  15& 
  254\\
   0.05    & 0.2&   0.35 & 0.4 & 10 &  1.2817 & 6 & 
   35& 
   974\\
     0.02    & 0.31&   0.33 & 0.34 & 10 &  1469.872 & 4 & 
471&  
6050\\
  0.02    & 0.31&   0.33 & 0.34 & 4 &  1709.399 & 1 & 
  470 & 
  6042\\
 \hline
\end{tabular}
\end{center}
\end{table}

\begin{table}[H]
\begin{center}
\caption{Approximations of the expected saved costs $s(\pi_{\mu})$, thresholds of optimal policies, values of $k$ and number of iterations for various instances with quadratic holding costs, $\mu=\mu_1,  K_{sqd}=1 $ and $ \epsilon=0.001$.}
\label{table6}
\begin{tabular}{ |p{0.8cm}|p{0.8cm}|p{0.8cm}|p{0.8cm}|p{0.8cm}|p{1.5cm}|p{1.5cm}|p{0.8cm}|p{1.4cm}|  }
 \hline
 \multicolumn{5}{|c|}{Parameters and costs} & \multicolumn{4}{|c|}{ Results}\\
 \hline
  $\beta$ & $\lambda$ & $\mu_1$ & $\mu_2$ & $c_{\mu_2}$ & $\tilde{s}_\epsilon(\pi_{\mu})$ & Threshold & $k$ & Iterations \\
 \hline
  0.1    & 0.1&   0.35 & 0.45 & 10 & 0.052 & 4 &
  16& 
  282\\
   0.05    & 0.2&   0.35 & 0.4 & 10 & 9.917 & 4& 
39 & 
1079\\
  0.02    & 0.31&   0.33 & 0.34 & 10 & 23307.671 & 0& 
  526
& 
6734\\
  0.02    & 0.31&   0.33 & 0.34 & 4 & 23587.750 & 0& 
  526 &
  6734\\
 \hline
\end{tabular}
\end{center}
\end{table}

 As before, the expected saved cost seems to be much higher when $\rho_\mu$ is closer to 1. The threshold for the case $\beta=0.05, \lambda =0.2,\mu=\mu_1=0.35,\mu_2=0.4,c_{\mu_2}=10$ with linear holding costs was infinite for discounts $\alpha=0.01$ and $\alpha=0.005$ (and therefore had 0 saved cost). Interestingly, in the non-discounted case, the threshold decreased from infinity to 4 and the saved cost increased from 0 to the high value $1469.872$. Also for quadratic holding costs there was the largest decrease in threshold and a very high increase in saved costs for this set of parameters. We suspect that this is due to the low value of $\beta$ and the increase of the saved cost to be assisted by the high value of $\rho_\mu$.

For the fixed rate $\mu=\mu_2$, we display a graph in Figure~\ref{fig:plot22} with a plot of the expected saved costs for the same specific combination of parameters as before.

\begin{figure}[H]
\centering
\begin{subfigure}{.5\textwidth}
  \centering
  \includegraphics[width=1\linewidth]{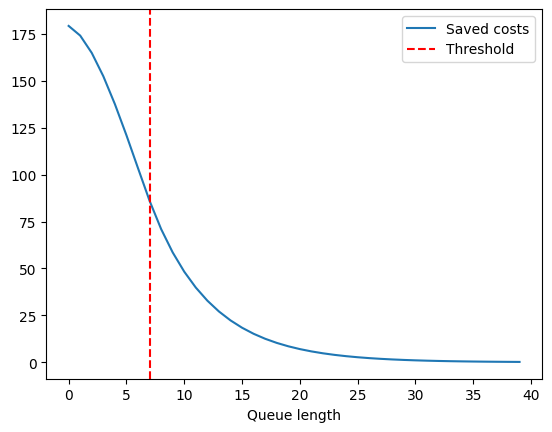}
\end{subfigure}%
\begin{subfigure}{.5\textwidth}
  \centering
  \includegraphics[width=1\linewidth]{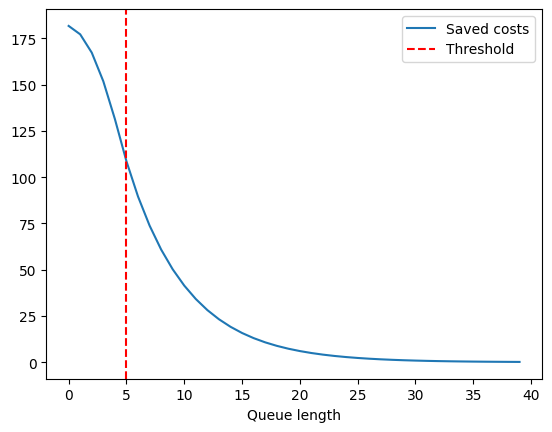}
\end{subfigure}
\caption{Plots of (approximated) expected saved costs ($\tilde{s}_\epsilon(i)$) for $ \beta=0.05, \lambda =0.2,\mu_1=0.35,\mu_2=0.4,c_{\mu_2}=10,\mu=\mu_2$ and $\epsilon=0.001$. The two cases correspond to linear holding cost function $h(i)=5\cdot i$, and quadratic holding cost function $h(i)=i^2$, respectively.}
\label{fig:plot22}
\end{figure}

We summarise the results for linear holding costs in Table~\ref{table7}  and for quadratic holding costs in Table~\ref{table8}. Again, Remark~\ref{optpolnondiscount} validates that the found approximating policies are optimal.

\begin{table}[H]
\begin{center}
\caption{Approximations of the expected saved costs $s(\pi_{\mu})$, thresholds of optimal policies, values of $k$ and number of iterations for various instances with linear holding costs, $\mu=\mu_2 , K_{lin}=5 $ and $ \epsilon=0.001$.}
\label{table7}
\begin{tabular}{ |p{0.8cm}|p{0.8cm}|p{0.8cm}|p{0.8cm}|p{0.8cm}|p{1.5cm}|p{1.5cm}|p{0.8cm}|p{1.4cm}|  }
 \hline
 \multicolumn{5}{|c|}{Parameters and costs} & \multicolumn{4}{|c|}{ Results}\\
 \hline
  $\beta$ & $\lambda$ & $\mu_1$ & $\mu_2$ & $c_{\mu_2}$ &  $\tilde{s}_\epsilon(\pi_{\mu})$ & Threshold & $k$ & Iterations \\
 \hline
  0.1    & 0.1&   0.35 & 0.45 & 10 &   94.912 & 6 & 12 & 
  250\\
   0.05    & 0.2&   0.35 & 0.4 & 10 & 170.946 & 7 & 
   28& 
   955\\
  0.02    & 0.31&   0.33 & 0.34 & 10 & 129.325 & 5 & 
  280
& 
5362\\
  0.01    & 0.33&   0.3 & 0.36 & 10 &  83.333 & 0 & 
  285
& 
10663\\
 \hline
\end{tabular}
\end{center}
\end{table}

\begin{table}[H]
\begin{center}
\caption{Approximations of the expected saved costs $s(\pi_{\mu})$, thresholds of optimal policies, values of $k$ and number of iterations for various instances with quadratic holding costs, $\mu=\mu_2 , K_{sqd}=1 $ and $ \epsilon=0.001$.}
\label{table8}
\begin{tabular}{ 
|p{0.8cm}|p{0.8cm}|p{0.8cm}|p{0.8cm}|p{0.8cm}|p{1.5cm}|p{1.5cm}|p{0.8cm}|p{1.4cm}|  }
 \hline
 \multicolumn{5}{|c|}{Parameters and costs} & \multicolumn{4}{|c|}{ Results}\\
 \hline
  $\beta$ & $\lambda$ & $\mu_1$ & $\mu_2$ & $c_{\mu_2}$ &  $\tilde{s}_\epsilon(\pi_{\mu})$ & Threshold & $k$ & Iterations \\
 \hline
  0.1    & 0.1&   0.35 & 0.45 & 10 &  97.721 & 5 & 
  13& 
  277\\
   0.05    & 0.2&   0.35 & 0.4 & 10 & 172.615 & 5 & 30
& 
1052\\
  0.02    & 0.31&   0.33 & 0.34 & 10 & 54.391 & 1 & 
  314
& 
5980
 \\
  0.01    & 0.33&   0.3 & 0.36 & 10 & 83.333 & 0 & 
  328
& 
12191\\
 \hline
\end{tabular}
\end{center}
\end{table}

When the fixed service rate is $\mu_2$, we observe smaller differences between discounted and non-discounted costs. This is because costs are saved in the earlier stages of the process (and higher holding costs occur later in the process). 
Also, we observe that the saved costs are equal between cases with the same parameters but different holding costs when the threshold is zero for both cases. This is because the saved cost is exactly the cost saved by using $\mu_1$ in state $(0,1)$ instead of using $\mu_2$ at cost $c_{\mu_2}$.

\bibliographystyle{amsrefs}
\bibliography{refs}

\appendix

\section{Equality of total saved cost for the continuous and discrete time models}\label{App:MuEE}

In this section we show the equivalence of Equations~(\ref{questionMathematicallyDiscounted}) and (\ref{questionMathematicallyDiscounted2}) for zero discounts. As there is equivalence for non-zero discounts, we only need to show that 
\begin{equation}\label{questionMathematicallyDiscounted3}
   \mathbb{E}^\phi_{p^0 \times p^1} \Bigg[ \int_0^\infty \Big( c(X(t))- c(Y(t))\Big) dt\Bigg]=\lim_{\alpha\downarrow 0}
    \mathbb{E}^\phi_{p^0 \times p^1} \Bigg[ \int_0^\infty e^{-\alpha t}\Big( c(X(t))- c(Y(t))\Big) dt\Bigg] .\end{equation}  
For this matter, we note that \cite{FloskeRecurrence} with Theorem~\ref{convergenceVIM} gives that the continuous time model is $M$-exponentially ergodic. This means that there exist constants $K_1,K_2,\nu> 0$ such that
\begin{equation*}
    \begin{split}    \Bigg|\mathbb{E}^\phi_{p_1}\Bigg[\int_{t=0}^\infty \Big(c(X_t)-\pi_\mu \cdot c\Big)   dt\Bigg]\Bigg|&\leq \int_{t=0}^\infty \Big|P^\phi_t\cdot p_1-\pi_\mu\Big|^{\top} \cdot c\    dt\leq  K_1\int_{t=0}^\infty \Big|P^\phi_t\cdot p_1-\pi_\mu\Big|^{\top} \cdot M\Big|   dt \\
    &\leq K_2\cdot \int_{t=0}^\infty e^{-\nu t}  dt<\infty.
    \end{split}
\end{equation*}
Together with a similar derivation for the total cost difference between starting distributions $p_0$ and $\pi_\mu$ and the dominated convergence theorem, Equation~(\ref{questionMathematicallyDiscounted3}) can be derived.

\section{Upper bound $u_M$}\label{App:uM}

In this section we want to determine an upper bound $u_M$ of $ \sup_{\phi} ||c^\phi||_M$.  
In the discounted case we have $0< \alpha < 1$ and $M$ defined by $M(i,q)=\gamma^i$ with $1<\gamma<1/(1-\alpha)$ for $(i,q)\in S$. Thus,  $$ \sup_{\phi} ||c^\phi||_M =\sup_{(i,q), \phi} \frac{|c^\phi(i,q)|}{M(i,q)} \leq \sup_{i} \frac{|c_{\mu_2}+h(i)|}{\gamma^i}.$$ 

For the non-discounted case, the function $M$ is defined by $M(i,0)= \cdot \zeta^i/2, M(i,1)= \zeta^i,$ with $\zeta\in \mathbb{R}_{\geq 1}$ for $i\in \mathbb{N}_0$. As such, $M$ is
bounded by $M(i,q)\geq \zeta^i/2$  for $(i,q)\in S$. This means that  
$$ \sup_{\phi} ||c^\phi||_M =\sup_{(i,q), \phi} \frac{|c^\phi(i,q)|}{M(i,q)} \leq 2 \sup_{i} \frac{|c_{\mu_2}+h(i)|}{\zeta^i} .$$ Thus, the expression for $u_M$ in the non-discounted case will be taken as twice the value of $u_M$ from the discounted case with $\gamma=\zeta$. Therefore, we will consider the value of $u_M$ for the discounted case in the rest of this section.

We will give $u_M$ as a function of $\gamma$.
 By Assumption~\ref{holdingassump} we can bound the holding costs by $h(i)\leq K\cdot 2^{i^{\epsilon}}$, with $0<\epsilon<1$ and $K\in \mathbb{R}$. 
 We can bound \begin{eqnarray*}
   \sup_{\phi} ||c^\phi||_M =\sup_{(i,q), \phi} \frac{|c^\phi(i,q)|}{M(i,q)} \leq  \sup_{i} \frac{|c_{\mu_2}+h(i)|}{\gamma^i} \leq c_{\mu_2} + K\cdot \sup_{i} \frac{ 2^{i^\epsilon}}{\gamma^i}.
\end{eqnarray*} 
To find $\sup_{i} 2^{i^\epsilon}\gamma^{-i}$, we look at the derivative of the real function $  2^{x^\epsilon}\gamma^{-x}$ with respect to $x$ for $x>0$. We use the product rule to find
\begin{equation*}
    \frac{d}{dx} \frac{ 2^{x^\epsilon}}{\gamma^x}=  \frac{ \frac{d}{dx}2^{x^\epsilon}}{\gamma^x}+2^{x^\epsilon}\cdot\frac{d}{dx}\frac{ 1}{\gamma^x}=\frac{\epsilon x^{\epsilon-1}\log(2)2^{x^\epsilon}}{\gamma^x}+\frac{2^{x^\epsilon}\cdot (-\log(\gamma))}{\gamma^x}= \frac{2^{x^\epsilon}(\epsilon x^{\epsilon-1}\log(2)-\log(\gamma))}{\gamma^x}.
\end{equation*}
This derivative is zero when
\begin{equation*}
    \epsilon x^{\epsilon-1}\log(2)-\log(\gamma)=0\Longleftrightarrow x^{1-\epsilon}=\frac{ \epsilon \log(2)}{\log(\gamma)} \Longleftrightarrow  x=\Big(\frac{ \epsilon \log(2)}{\log(\gamma)}\Big)^{1/(1-\epsilon)}.
\end{equation*}
As $  \epsilon x^{\epsilon-1}\log(2)-\log(\gamma)$ is decreasing in $x$ for $x>0$, we note that the derivative of $  2^{x^\epsilon}\gamma^{-x}$ is positive for $0<x< ( \epsilon \log(2)/\log(\gamma))^{1/(1-\epsilon)}$ and negative for $x> ( \epsilon \log(2)/\log(\gamma))^{1/(1-\epsilon)}$.
Therefore, the maximal value of $2^{x^\epsilon}\gamma^{-x}$ is attained at $x= ( \epsilon \log(2)/\log(\gamma))^{1/(1-\epsilon)}$. Thus,
 \begin{equation*}
   \sup_{i} \frac{2^{i^\epsilon}}{\gamma^i}\leq \frac{2^{\Big(\frac{ \epsilon \log(2)}{\log(\gamma)}\Big)^{\epsilon/(1-\epsilon)}}}{\gamma^{\Big(\frac{ \epsilon \log(2)}{\log(\gamma)}\Big)^{1/(1-\epsilon)}}}, \ \ \ \text{yielding an upper bound of } \ u_M=c_{\mu_2}+K\cdot \frac{2^{\Big(\frac{ \epsilon \log(2)}{\log(\gamma)}\Big)^{\epsilon/(1-\epsilon)}}}{\gamma^{\Big(\frac{ \epsilon \log(2)}{\log(\gamma)}\Big)^{1/(1-\epsilon)}}}.
 \end{equation*}
  A tighter bound on $h$ gives a tighter upper bound $u_M$ that grows slower when $\gamma\downarrow 1$. The result will be tighter convergence guarantees for VI and our approximations of the expected saved costs.
That is why we also consider finding a sufficient $u_M$ for polynomially bounded holding costs, i.e., $h(i)\leq  K\cdot i^d$ with $d\in \mathbb{N}$. Now,
 \begin{eqnarray*}
   \sup_{\phi} ||c^\phi||_M \ \leq c_{\mu_2} + K\cdot \sup_{i} \frac{ i^d}{\gamma^i}.
\end{eqnarray*} 
The derivative of $\frac{x^d}{\gamma^x}$ for $x>0$ is given by
\begin{equation*}
    \frac{d}{dx} \frac{ x^d}{\gamma^x}=  \frac{ \frac{d}{dx}  x^{d}}{\gamma^x}+x^d\cdot\frac{d}{dx}\frac{ 1}{\gamma^x}=\frac{dx^{d-1}}{\gamma^x}+\frac{x^d\cdot (-\log(\gamma))}{\gamma^x}= \frac{x^{d-1}(d-x\log(\gamma))}{\gamma^x}.
\end{equation*}
This derivative is zero when $x\log(\gamma)=d \Leftrightarrow x=d/\log(\gamma)$, positive for $0<x<d/\log(\gamma)$ and negative for $x>d/\log(\gamma)$. Thus, as $\gamma^{1/\log (\gamma)}=e$,
\begin{equation*}\label{eq:supje}
     \sup_{i} \frac{i^d}{\gamma^i}\leq \frac{\Big(\frac{d}{\log(\gamma)}\Big)^d}{\gamma^{\frac{d}{\log(\gamma)}}}=\frac{\Big(\frac{d^d}{\log^d(\gamma)}\Big)}{e^d}=\frac{d^d}{\log^d(\gamma)e^d}.
\end{equation*}
This bound yields sufficient upper bound 
\begin{equation}\label{eq:umpol}
u_M=c_{\mu_2}+\frac{K\cdot d^d}{\log^d(\gamma)\cdot e^d}   . 
\end{equation}

Applying the substitution of $\gamma$ by $\zeta$ and multiplication by 2 gives the following value of $u_M$ for the non-discounted case with polynomially bounded holding costs 
\begin{equation}\label{eq:umpol2}
u_M=2\Big(c_{\mu_2}+\frac{K\cdot d^d}{\log^d(\zeta)\cdot e^d} \Big)  . 
\end{equation}

\section{The expected total discounted cost for a specific policy}\label{sect:specpol}

Let $\phi \in \Phi$ be a given policy. We will determine an $\epsilon$-approximation of the expected saved cost (or extra cost) by using policy $\phi$ and starting with control instead of never having control. We do this by analysing the expected cost difference during control and the expected cost difference after losing control, as not only the sum, but both parts themselves can be of interest.

For the states $(i,1)$, we let $\tau_0$ be the time step where queueing control is lost (note that $\tau_0$ is independent of the queue length). We let $\{X_n\}_{n\in \mathbb{N}_0}$ be the MDP using policy $\phi\in \Phi$.

As such, we can write \begin{equation}
     V^\phi_{\alpha}(i,1)= \mathbb{E}_{(i,1)}^\phi \Big[ \sum_{n=0}^\infty (1-\alpha)^n c(X_n) \Big]=\mathbb{E}_{(i,1)}^\phi \Big[ \sum_{n=0}^{\tau_0-1} (1-\alpha)^n c(X_n)+\sum_{n=\tau_0}^{\infty} (1-\alpha)^n c^\phi(X_n) \Big] .\label{split0}
\end{equation}

We split the expected value of Equation (\ref{split0}) in two parts and through conditioning on $\tau_0$, we find
\begin{align}
\begin{split}
   \mathbb{E}_{(i,1)}^\phi \Big[ \sum_{n=0}^{\tau_0-1} (1-\alpha)^n c(X_n)\Big] &= \sum_{k=0}^\infty \mathbb{P}(\tau_0=k+1)\sum_{n=0}^{k} (1-\alpha)^n \mathbb{E}_{(i,1)}^\phi \Big[ c(X_n) \ |\ \tau_0=k+1\Big]\\
   &= \sum_{k=0}^\infty \mathbb{P}(\tau_0=k+1)\sum_{n=0}^{k} (1-\alpha)^n \mathbb{E}_{(i,1)}^\phi \Big[ c(X_n) \ |\ \tau_0\geq n+1\Big]\\
   &=\sum_{k=0}^\infty \Big( \sum_{n=k}^\infty \mathbb{P}(\tau_0 = n+1) \Big)  (1-\alpha)^{k} \mathbb{E}_{(i,1)}^\phi \Big[ c(X_{k} ) \ |\ \tau_0\geq k+1\Big]\\
 &=\sum_{k=0}^\infty \mathbb{P}(\tau_0\geq k+1) (1-\alpha)^{k} \mathbb{E}_{(i,1)}^\phi \Big[ c(X_{k} ) \ |\ \tau_0\geq k+1\Big]\\
 &=\sum_{k=0}^\infty ((1-\beta) (1-\alpha))^{k} \mathbb{E}_{(i,1)}^\phi \Big[ c(X_{k} ) \ |\ \tau_0\geq k+1\Big].
 \end{split}
 \label{split1}
\end{align}

We observe that the final sum of Equation (\ref{split1}) corresponds with the discrete time queueing control model with discount $1-(1-\beta)(1-\alpha)=\beta+\alpha-\alpha \beta$, and transition probabilities that are scaled by $1-\beta$, as we condition on queuing control not getting lost.
As such, for $\phi$ with $\phi(i,1)=a$, the non-zero elements of transition matrix ${P^\ast}^{\phi}$ are given by

$$ {p^\ast}^{\phi}_{(i,1),(j,1)} = \begin{cases} \frac{\lambda}{1-\beta} \ \ \ \ \  
&\text{for } j=i+1,\\
\frac{a}{1-\beta} \mathds{1}_{\{i>0\}} \ \ \ \ \ & \text{for } j=i-1,\\
1-\frac{\lambda+a\mathds{1}_{\{i>0\}}}{1-\beta} \ \ \ \ \ &  \text{for } j=i.\end{cases}$$

The expected (discounted) cost of this model can be denoted as ${V^\ast}^\phi_{(\alpha+\beta-\alpha \beta)}(i)$. We can approximate this value as Assumption~\ref{assump2alpha} is also satisfied here, in the same way as the original discounted model satisfies Assumption~\ref{assump2alpha} with $M(i,q)=\gamma^i$.
The current equivalent of Equation~(\ref{eq:vfan1}) now gives
\begin{equation} \label{eq:specdisc}
     |{V^\ast}^\phi_{(\alpha+\beta-\alpha \beta)}(i)-{V^\ast}^\phi_{(\alpha+\beta-\alpha \beta),n}(i)| \leq \frac{(1-(\alpha+\beta-\alpha \beta))^n\gamma^n u_M}{1-(1-(\alpha+\beta-\alpha \beta))\gamma} \cdot \gamma^i.
\end{equation}

Then, we consider the second part of Equation (\ref{split0}) and, conditioning on $\tau_0$, we find
\begin{align*}
\begin{split}
   \mathbb{E}_{(i,1)}^\phi \Big[\sum_{n=\tau_0}^{\infty} (1-\alpha)^n c(X_n) \Big]&= \sum_{k=1}^\infty \mathbb{P}(\tau_0=k)(1-\alpha)^k \sum_{n=0}^\infty  (1-\alpha)^n\mathbb{E}^\phi_{(i,1)}\Big[ c(X_{k+n} ) \ |\ \tau_0=k \Big]\\
   &= \sum_{k=1}^\infty \beta (1-\beta)^{k-1} (1-\alpha)^k \sum_{n=0}^\infty (1-\alpha)^n \mathbb{E}^\phi_{(i,1)}\Big[ c(X_{k+n} ) \ |\ \tau_0=k \Big].
\end{split}
\end{align*}
We can consider the probability distribution over queue length at time $\tau_0$ as $\pi^{\ast,\phi,k}_i$ where \\
$\pi^{\ast,\phi,0}_i(j)= \mathds{1}_{i=j} $ and
$\pi^{\ast,\phi,k}_i(j)=({P^\ast}^{\phi})^{k-1}_{i,j}$ for $k\geq 1$. Then,

\begin{align*}
\begin{split}
 & \mathbb{E}_{(i,1)}^\phi \Big[\sum_{n=\tau_0}^{\infty} (1-\alpha)^n c(X_n ) \Big]= \sum_{k=1}^\infty \beta (1-\beta)^{k-1} (1-\alpha)^k \sum_{n=0}^\infty (1-\alpha)^n \mathbb{E}^\phi_{(i,1)}\Big[ c(X_{k+n}) \ |\ \tau_0=k \Big] \\
  & \hspace{4.2cm}=   \sum_{k=1}^\infty \beta (1-\beta)^{k-1} (1-\alpha)^k \sum_{n=0}^\infty (1-\alpha)^n \sum_{j \in \N_0} \pi^{\phi,\ast,k}_i(j)  \mathbb{E}^\phi_{(j,0)}\Big[ c(X_{n} )  \Big] \\
  &\hspace{4.2cm}= \sum_{k=1}^\infty \beta (1-\beta)^{k-1} (1-\alpha)^k \sum_{j \in \N_0} \pi^{\phi,\ast,k}_i(j) \sum_{n=0}^\infty \mathbb{E}^{\phi_0}_{(j,0)}\Big[  (1-\alpha)^n c(X_{n})\Big] \\
  &\hspace{4.2cm}=  \sum_{k=1}^\infty \beta (1-\beta)^{k-1} (1-\alpha)^k \sum_{j \in \N} \pi^{\phi,\ast,k}_i(j) V_{\alpha}(j,0) .
\end{split}
\end{align*}
With notation $ V_{\alpha}(\pi,0)=\sum_{i\in \mathbb{N}_0} \pi(i) V_{\alpha}(i,0)$, we can write
\begin{align*}
     \sum_{k=1}^\infty \beta (1-\beta)^{k-1} (1-\alpha)^k \sum_{j \in \N} \pi^{\ast,\phi,k}_i(j) V_{\alpha}(j,0) = \sum_{k=1}^\infty \beta (1-\beta)^{k-1} (1-\alpha)^k  V_{\alpha}(\pi^{\ast,\phi,k}_i,0) .
\end{align*}
This is not a tractable sum. As such, we need to bound the tail. We can do this with a bound on $V_{\alpha}$.
We note that for $\phi_0$ being the policy only using $\mu$, we have that
$ \inf_{\phi} V_\alpha^{\phi}(i,1)\leq V_\alpha^{\phi_0}(i,1)= V_\alpha(i,0)$.\\
By Equation (\ref{eq:vfan2}),
\begin{align*}
     V_{\alpha}(\pi^{\ast,\phi,k}_i,0) \leq   V_{\alpha}(i+k,0) \leq  \frac{u_M \gamma^{i+k} }{1-(1-\alpha)\gamma} .
\end{align*}
Looking at the tail of $ \mathbb{E}_{(i,1)}^\phi \Big[\sum_{n=\tau_0}^{\infty} (1-\alpha)^n c(X_n) \Big]= \sum_{k=1}^\infty \beta (1-\beta)^{k-1} (1-\alpha)^k  V_{\alpha}(\pi^{\ast,\phi,k}_i,0)$ we get

\begin{equation}\label{eq:spectail}
\begin{split}
    &\sum_{k=n}^\infty \beta (1-\beta)^{k-1} (1-\alpha)^k  V_{\alpha}(\pi^{\ast,\phi,k}_i,0) \leq  \sum_{k=n}^\infty \beta (1-\beta)^{k-1} (1-\alpha)^k \frac{u_M \gamma^{i+k} }{1-(1-\alpha)\gamma} \\
    &\hspace{3cm}= \frac{u_M \beta \gamma^i }{ 1-(1-\alpha)\gamma} \sum_{k=n}^\infty  (1-\beta)^{k-1} (1-\alpha)^k  \leq  \frac{u_M \beta \gamma^i}{1-(1-\alpha)\gamma} \sum_{k=n}^\infty  (1-\beta)^{k-1} \\
   & \hspace{3cm}= \frac{u_M \beta \gamma^i}{1-(1-\alpha)\gamma}\cdot \frac{(1-\beta)^{n-1}}{\beta} = \frac{u_M \gamma^i (1-\beta)^{n-1} }{1-(1-\alpha)\gamma}.
    \end{split}
\end{equation}

The bounds from  Equations~(\ref{eq:specdisc}) and (\ref{eq:spectail}) yield an $\epsilon$-approximation of $V^\phi_{\alpha}(i,1)$.



\begin{proposition} \label{thmspecpolicy}
We can $\epsilon$-approximate $V^\phi_{\alpha}(i,1)$ 
 by \begin{equation*}
    \Tilde{V}^\phi_{\alpha, \epsilon}(i,1)= {V^\ast}^\phi_{(\alpha+\beta-\alpha \beta),n^\ast_{\epsilon,i}}(i)+ \sum_{k=1}^{\ell^\ast_{\epsilon,i}} \beta (1-\beta)^{k-1} (1-\alpha)^k  V_{\alpha}(\pi^{\ast,\phi,k}_i,0),
\end{equation*}
with \begin{equation*}
    n_{\epsilon,i}^\ast=\bigg\lceil\log_{(1-(\alpha+\beta-\alpha \beta))\gamma}\bigg( \frac{1-(1-(\alpha+\beta-\alpha \beta))\gamma}{u_M \gamma^i}\bigg)\bigg\rceil,
\end{equation*}
and \begin{equation*}
    \ell_{\epsilon,i}^\ast=\bigg\lceil\log_{1-\beta}\bigg( \frac{1-(1-\alpha)\gamma}{u_M \gamma^i}\bigg)\bigg\rceil+1.
\end{equation*}
 
\end{proposition}
\begin{proof}
 Equation~(\ref{eq:specdisc}) gives that
$ |{V^\ast}^\phi_{(\alpha+\beta-\alpha \beta)}(i)-{V^\ast}^\phi_{(\alpha+\beta-\alpha \beta),n_\epsilon^\ast}(i)| \leq \epsilon/2$ and Equation~(\ref{eq:spectail}) gives that $ \sum_{k=\ell_\epsilon^\ast}^\infty \beta (1-\beta)^{k-1} (1-\alpha)^k  V_{\alpha}(\pi^{\ast,\phi,k}_i,0) \leq \epsilon/2$. Combining the two parts of Equation~(\ref{split0}) gives that $|\Tilde{V}^\phi_{\alpha,\epsilon}(i,1) - V^\phi_{\alpha}(i,1)|\leq \epsilon/2+\epsilon/2 =\epsilon $.

\end{proof}

\section{Verification of  $M$-uniform geometric recurrence}\label{app:Mrec}

In this section, we will show that there exist $\zeta>1$, and $\eta<1$, such that 
the function $M$ defined by $$M(i,0)=\frac{1}{2} \cdot \zeta^i, M(i,1)= \zeta^i,\quad i=0,1,2,\ldots$$ 
satisfies Eqn.~(\ref{eq:M}), for  $D=\{(0,0)\}$, i.o.w., such that $\sum_{y\not=(0,0)}p^\phi_{xy}M(y)\leq \eta M(x)$, $x\in S$.

First consider states $x=(i,0)\in S$ with $i\geq 2$. Then,
\begin{eqnarray*}
\lefteqn{\hspace*{-4cm} 
\sum_{ y\not=(0,0)}p^\phi_{xy}M(y)= \mu M(i-1,0)+(1-\lambda-\mu)M(i,0)+\lambda M(i+1,0)}\\[-8pt]
 &=&\frac{1}{2}\zeta^i \cdot \Big(\frac{\mu}{\zeta}+(1-\lambda-\mu)+\lambda \zeta\Big)\\
    &=&\Big(\frac{\mu}{\zeta}+(1-\lambda-\mu)+\lambda \zeta\Big)M(i,0).
\end{eqnarray*}
For notational convenience, we write $f_\mu(\zeta):=\frac{\mu}{\zeta}+(1-\lambda-\mu)+\lambda \zeta$.

Clearly, for state $x=(1,0)$, $\sum_{y\not=(0,0)}p^\phi_{(1,0)y}M(y)\leq f_{\mu}(z) M(1,0)$, as we take the summation over less states.
Hence, for Eqn.~(\ref{eq:M}) to hold for states $(i,0)$, $i\geq1$, we need that
$f_\mu(\zeta)<1$, for some $\zeta>1$. It is easily verified, that 
\begin{equation}
\label{eq: cond01}
f_\mu(\zeta)<1, \quad \mbox{for } \zeta\in(1,\mu/\lambda).
\end{equation}
Subsequently,  
for initial state $(0,0)$, we get
$$
\sum_{y\not=(0,0)}p^\phi_{(0,0)y}M(y)=\lambda  M(1,0)=\lambda\cdot\zeta M(0,0),
$$
so that we need
\begin{equation}
\label{eq: cond02}
\zeta< \frac{1}{\lambda}.
\end{equation}
Next,   using  $\mu_1<\mu_2$ and increasingness of $M(i,1)$ in $i$  for states $(i,1)$, $i\geq 1$,  yields
\begin{eqnarray} \label{eq: cond2}
\lefteqn{\hspace*{-1cm}  \sum_{y\not=(0,0)}p^\phi_{(i,1)y}M(y) }\nonumber\\[-5pt]
&=&\phi(i,1) M(i-1,1)+\beta M(i,0)+(1-\lambda-\beta-\phi(i,1))M(i,1)+\lambda M(i+1,1)\nonumber\\ 
    &\leq& \mu_1M(i-1,1)+\beta M(i,0)+(1-\lambda-\beta-\mu_1)M(i,1)+\lambda M(i+1,1)\nonumber\\
    &=&\Big(\frac{\mu_1}{\zeta}+\frac{1}{2}\beta +(1-\lambda-\beta-\mu_1)+\lambda \zeta\Big)\cdot\zeta^i\nonumber\\
    &=&\Big(\frac{\mu_1}{\zeta} +(1-\lambda-\frac{1}{2}\beta-\mu_1)+\lambda \zeta\Big)M(i,1)=\Big (f_{\mu_1}(\zeta)-\frac{\beta}{2}\Big) M(i,1)\leq f_{\mu_1}(\zeta) M(i,1).
\end{eqnarray}
By virtue of Eqn.~(\ref{eq: cond01}),  $f_{\mu_1}(\zeta)<1$, for $\zeta\in (1,\mu_1/\lambda)$, if $\mu_1>\lambda$.

Suppose, that $\mu_1\leq \lambda <\mu_2=\mu$. 
Then,  the function $ \zeta\rightarrow f_{\mu_1}(\zeta)-\beta/2$ has a minimum at $\zeta=\sqrt{{\mu_1}/{\lambda}}<1$, and is smaller than 1 at $\zeta=1$. 
Computing the two roots of the equation $f_{\mu_1}(\zeta)=1$ thus yields, that $f(\mu_1)(\zeta)-\frac{\beta}{2}<1$ for $\mu_1\leq \lambda$, whenever 
\begin{equation}
\label{eq: cond3}
 \zeta \in\Bigg(1, \frac{(\lambda+\frac{1}{2}\beta+\mu_1)+\sqrt{(\lambda+\frac{1}{2}\beta+\mu_1)^2-4\lambda \mu_1}}{2\lambda}\Bigg).
\end{equation}
The final case to consider is state $x=(0,1)$, for which
\begin{eqnarray*}
\sum_{y\not=(0,0)}p_{(0,1)y}M(y)=(1-\lambda-\beta)M(0,1)+\lambda M(1,1) = (1-\lambda-\beta+\lambda \zeta)M(0,1),
\end{eqnarray*}
so that we need
\begin{equation}
\label{eq: cond4}
\zeta< 1+\frac{\beta}{\lambda}.
\end{equation}
Summarising, combining (\ref{eq: cond01}), (\ref{eq: cond02}), (\ref{eq: cond2}),  (\ref{eq: cond3}) and (\ref{eq: cond4}), yields, that Eqn.~(\ref{eq:M}) holds 
when $\mu_1>\lambda$, for $\zeta$ satisfying
$$
1<\zeta < \frac{\mu_1}{\lambda}\wedge \big(1+\frac{\beta}{\lambda}\big),
$$
and when $\mu_1\leq \lambda<\mu=\mu_2$, for $\zeta$ satisfying
$$
1<\zeta< \big(1+\frac{\beta}{\lambda}\big)\wedge \frac{\mu}{\lambda}\wedge \frac{(\lambda+\frac{1}{2}\beta+\mu_1)+\sqrt{(\lambda+\frac{1}{2}\beta+\mu_1)^2-4\lambda \mu_1}}{2\lambda},
$$
and, given $\zeta$, for $\eta$ satisfying 
$$
\eta= \big(1-\lambda -\beta+\lambda\zeta\big) \vee f_{\mu}(\zeta)\vee \big( f_{\mu_1}(\zeta)-\frac{\beta}{2}\big).
$$

\section{Proofs with $M$-uniform geometric recurrence}\label{sect:MProofs}

An alternative way to approximate $s(i)$ (and $s(\pi_\mu)$ implicitly through the use of the different approximations) is to use $M$-uniform geometric recurrence, instead of using the equivalence to a $\beta-$discounted MDP.  The weakness of $M$-uniform geometric recurrence in practice for these approximations occurs through a high number of iterations when $\eta$ must be close to 1. When this problem does not occur, the method is effective. Furthermore, we also consider this method to be of theoretical interest.

We now will give equivalent statements of Lemma~\ref{HAppr} and  Proposition~\ref{corolnondiscounti}, which we then prove using $M$-uniform geometric recurrence.
The proofs of Theorem~\ref{NonDiscountFind} and Proposition~\ref{corolnondiscount} are only affected implicitly through the changes in  Lemma~\ref{HAppr} and  Proposition~\ref{corolnondiscounti}.

\begin{restatable}{lemma}{HApproximate}\label{HAppr2}
  Let $\epsilon>0,k\in \mathbb{N}_0$,
 we can determine $\epsilon$-approximations \\
$\tilde{H}_\epsilon(i,1)=v_{\Bar{n}_{\epsilon,k}}(i,1),0\leq i \leq k$ of the values $H^\ast(0,1),H^\ast(1,1),\dots,H^\ast(k,1)$. 
Here, vectors $v_n,n\in \mathbb{N}_0$ are iteratively constructed with $v_0\equiv 0$ and 
\begin{equation}
    \begin{split} \label{eq:vnvect}
        v_{n+1}=\min_{\phi \in \Phi} \Big\{c^\phi-g^\ast+\prescript{}{(0,0)}P^\phi v_n \Big\}, 
    \end{split}
\end{equation}
and $\Bar{n}_{\epsilon,k}$ is given by
\begin{equation}
\begin{split}
\Bar{n}_{\epsilon,k}=\Bigg\lceil \log_\eta \Bigg(\frac{\epsilon\cdot \zeta^{-k}(1-\eta)}{u_M}\Bigg) \Bigg\rceil.
\end{split}
\end{equation}

\end{restatable}
\begin{proof}
     

Let $\phi^{n+1}$ denote the minimising policy in step $n+1$ of Equation~(\ref{eq:vnvect}). Equation~(\ref{eq:Hbound2}) gives that
 $$H^\ast-v_{n+1} \leq  c^{\phi^{n+1}}-g^\ast+\prescript{}{(0,0)}P^{\phi^{n+1}} H^\ast -(c^{\phi^{n+1}}-g^\ast+\prescript{}{(0,0)}P^{\phi^{n+1}} v_n) = \prescript{}{(0,0)}P^{\phi^{n+1}}(H^\ast-v_n),$$ and therefore 
$$H^\ast -v_n \leq  \Big(\Pi_{k=1}^n 
\ \prescript{}{(0,0)}P^{\phi^n} \Big)H^\ast \leq  \Big(\Pi_{k=1}^n 
\ \prescript{}{(0,0)}P^{\phi^n} \Big)\cdot \frac{u_M}{1-\eta} M \leq \frac{u_M \cdot \eta^n}{1-\eta} M.$$ 
Likewise, $$ v_{n+1}-H^\ast \leq c^{\phi^\ast}-g^\ast+\prescript{}{(0,0)}P^{\phi^\ast} v_n -(c^{\phi^\ast}-g^\ast+\prescript{}{(0,0)}P^{\phi^\ast} H^\ast) = \prescript{}{(0,0)}P^{\phi^\ast}(v_n-H^\ast),  $$ such that $$ v_{n}-H^\ast \leq ( \prescript{}{(0,0)}P^{\phi^\ast})^n |H^\ast| \leq ( \prescript{}{(0,0)}P^{\phi^\ast})^n\frac{u_M}{1-\eta} M  \leq  \frac{u_M \cdot \eta^n}{1-\eta} M.$$
Thus, we can conclude 
\begin{equation} \label{eq:vnapprox}
    ||H^\ast -v_n||_M\leq \frac{u_M \cdot \eta^n}{1-\eta} .
\end{equation}

For $0\leq i\leq k,q\in \{0,1\}$ we note that $M(i,q)\leq M(k,1)=\zeta^k$,
 hence,
Equation~(\ref{eq:vnapprox}) gives that $$|H^\ast(i,q)-v_{\Bar{n}_{\epsilon,k}}(i,q)|\leq  \frac{u_M \cdot \eta^{\Bar{n}_{\epsilon,k}} \cdot \zeta^k}{1-\eta}\leq \epsilon.$$
\end{proof}

When applying  Lemma~\ref{HAppr2}, it is also necessary to provide an alternative to Proposition~\ref{corolnondiscounti}. 

\begin{restatable}{proposition}{proponondiscounti}\label{corolnondiscounti2}
    Let $\epsilon>0,i\in \mathbb{N}_0$.
    The policy $\phi^{\hat{n}_{\epsilon,i}+1}\in \Phi$
    with
    $$ \hat{n}_{\epsilon,i}= 
    \Bigg\lceil \log_{\eta}\Bigg( \frac{\epsilon (1-\eta)}{2\zeta^i(u_M+2g^\ast)} \Bigg) \Bigg \rceil $$
    is $\epsilon$-optimal 
   for initial queue length $i\in \mathbb{N}_0$.
\end{restatable}
\begin{proof}
    We define $v'_{\hat{n}_{\epsilon,i}}:=v_{\hat{n}_{\epsilon,i}}$ and take policy $\phi^{\hat{n}_{\epsilon,i}+1}$ minimising Equation~(\ref{eq:vnvect}). 
    We define $H'$ as the relative value vector of the model resulting from VI when we restrict $\Phi$ to $\{\phi^{\hat{n}_{\epsilon,i}+1}\}$ with $H'(0,0)=0$. 
Theorem~\ref{convergenceVIM} applies, such that, equivalently to Equation~(\ref{eq:Hbound2}) also $| |H'||_M\leq u_M/(1-\eta) $.
Note that 
\begin{equation} \label{eq:Hdiff}
    ||H^\ast(i,1)-H'(i,1)||_M \leq ||H^\ast(i,1)-v_n(i,1)||_M+||H'(i,1)-v'_n(i,1)||_M.
\end{equation} 
Equation~(\ref{eq:vnapprox}) gives that 
\begin{equation}\label{eq:Hdiff1}
    ||H^\ast(i,1)-v_n(i,1)||_M \leq \frac{u_M \cdot \eta^n }{1-\eta}\leq \frac{\epsilon}{2} \cdot \zeta^{-i}. 
\end{equation}

For $m\geq \hat{n}_{\epsilon,i}$, we iteratively define 
\begin{equation}
    \begin{split} \label{eq:vnvect2}        v'_{m+1}=c^{\phi_{i,\epsilon}}-g^\ast+\prescript{}{(0,0)}P^{\phi_{i,\epsilon}} v'_m .
    \end{split}
\end{equation}
We note that $H'=c^{\phi_{i,\epsilon}}-g^\ast +\prescript{}{(0,0)}P^{\phi_{i,\epsilon}}H'$. As such, for $m\geq \hat{n}_{\epsilon,i}$, \begin{equation*}
    |H'-v'_{m+1}|= \prescript{}{(0,0)}P^{\phi_{i,\epsilon}}|H'-v'_{m}| \implies || H'-v'_{m+1} ||_M \leq \eta ||  H'-v'_{m}||_M.
\end{equation*}
 Also, $||v_{\hat{n}_{\epsilon,i}}||_M\leq || H^\ast-v_{\hat{n}_{\epsilon,i}} ||_M+||H^\ast||_M<\infty$ such that $||H'-v'_{\hat{n}_{\epsilon,i}}||_M\leq || H'||_M+||v_{\hat{n}_{\epsilon,i}}||_M<\infty $. Thus, $\lim_{m\rightarrow \infty} || H'-v'_m||_M=0$. 
 We note that $v'_{m+2}-v'_{m+1}=\prescript{}{(0,0)}P^{\phi_{i,\epsilon}} (v'_{m+1}-v'_m)$ for $m\geq \hat{n}_{\epsilon,i}$. Now,
 \begin{equation}
     \begin{split} \label{eq:Hvn}
         ||H'-v'_{\hat{n}_{\epsilon,i}}||_M &\leq \lim_{m\rightarrow \infty}  \Big(||H'-v'_m||_M+\sum_{j=0}^{m-1}|| v'_{\hat{n}_{\epsilon,i}+j+1} -v'_{\hat{n}_{\epsilon,i}+j}||_M \Big)= \sum_{j=0}^{\infty}|| v'_{\hat{n}_{\epsilon,i}+j+1} -v'_{\hat{n}_{\epsilon,i}+j}||_M \\
         &\leq \sum_{j=0}^{\infty} \eta^i || v'_{\hat{n}_{\epsilon,i}+1}-v'_{\hat{n}_{\epsilon,i}}||_M =\frac{|| v'_{\hat{n}_{\epsilon,i}+1}-v'_{\hat{n}_{\epsilon,i}}||_M}{1-\eta}=\frac{|| v_{\hat{n}_{\epsilon,i}+1}-v_{\hat{n}_{\epsilon,i}}||_M}{1-\eta}.
     \end{split}
 \end{equation}

 Going back to Equation~(\ref{eq:vnvect}) we note that for $1\leq m\leq \hat{n}_{\epsilon,i}$ we have 
 \begin{equation*}
     \begin{split}
         c^{\phi^{m+1}}-g^\ast + \prescript{}{(0,0)}P^{\phi^{m+1}} v_m -(c^{\phi^{m+1}}-g^\ast + \prescript{}{(0,0)}P^{\phi^{m+1}}v_{m-1}) = \prescript{}{(0,0)}P^{\phi^{m+1}}(v_m-v_{m-1}) \leq v_{m+1}-v_m \\
         \leq c^{\phi^{m}}-g^\ast + \prescript{}{(0,0)}P^{\phi^{m}} v_m -(c^{\phi^{m}}-g^\ast + \prescript{}{(0,0)}P^{\phi^{m}}v_{m-1}) = \prescript{}{(0,0)}P^{\phi^{m}}(v_m-v_{m-1}).
     \end{split}
 \end{equation*}  

 These bounds give that $ || v_{m+1}-v_m ||_M\leq \eta || v_{m}-v_{m-1} ||_M$ and therefore, by Equation~(\ref{eq:Hvn}) and the fact that $M\geq \frac{1}{2}$, 
 give that \begin{equation} \label{eq:Hdiff2}
     ||H'-v'_{\hat{n}_{\epsilon,i}}||_M\leq \frac{\eta^{\hat{n}_{\epsilon,i}}|| v_1-v_0 ||}{1-\eta} = \frac{\eta^{\hat{n}_{\epsilon,i}}|| c^{\phi^1}-g^\ast||_M}{1-\eta}\leq \frac{\eta^{\hat{n}_{\epsilon,i}} (u_M+2g^\ast)}{1-\eta}\leq \frac{\epsilon}{2} \cdot \zeta^{-i}.
 \end{equation}

 Equations~(\ref{eq:Hdiff}), (\ref{eq:Hdiff1}), (\ref{eq:Hdiff2}) and the fact that $M(i,1)=\zeta^i$ result in \begin{equation*}
      |H^\ast(i,1)-H'(i,1)| \leq (\frac{\epsilon}{2}\cdot\zeta^{-i}+\frac{\epsilon}{2}\cdot\zeta^{-i})\cdot\zeta^i =\epsilon.
 \end{equation*}

 We can conclude that 

\begin{equation*}
    \begin{split}
        \mathbb{E}^{\phi^{\hat{n}_{\epsilon,i}}}_{\delta_i^0\times\delta_i^1} \Bigg[ \sum_{n=0}^\infty \Big( c(X_n)- c(Y_n) \Big) \Bigg]&=H^\ast(i,0)-H'(i,1)\\
        &\geq H^\ast(i,0)-H^\ast(i,1)-  |H^\ast(i,1)-H'(i,1)|
\geq s(i)-\epsilon.
    \end{split}
\end{equation*}

\end{proof}

\section{The relative values $H_\mu$ for linear and quadratic holding costs}\label{sect:HfVer}

\subsection{Linear holding costs
}\label{sect:HfVerLin}

For linear holding cost function $h(i)=K_{lin}\cdot i$, we propose and verify the solution 
\begin{equation}\label{hmulin}
    H_{\mu}(i)=\frac{K_{lin}i(i+1)}{2(\mu-\lambda)}
\end{equation} of Equation~(\ref{Hiter}).
Recall $g=c_\mu+K_{lin} \rho_{\mu}/(1-\rho_{\mu})$.
 We see $H_{\mu}(0)=0$ and $$H_{\mu}(1)= \frac{K_{lin}}{\mu-\lambda}=\frac{K_{lin}}{\lambda(1/\rho_{\mu}-1)} = \frac{K_{lin}\rho_{\mu}}{\lambda(1-\rho_{\mu})} = \frac{g-c_\mu}{\lambda}.$$
From the formula we get $H_{\mu}(i)-H_{\mu}(i-1)=K_{lin}\cdot i/(\mu-\lambda) $. Now through induction we find 
\begin{eqnarray*}
    H_{\mu}(i+1)&=& H_{\mu}(i)+\frac{g-c_\mu-h(i)+\mu(H_{\mu}(i)-H_{\mu}(i-1))}{\lambda}\\
    &=&\frac{K_{lin}\cdot i(i+1)}{2(\mu-\lambda)}+\frac{K_{lin}}{\mu-\lambda}-\frac{K_{lin}\cdot i}{\lambda}+\frac{\mu K_{lin}\cdot i/\lambda}{\mu-\lambda}\\&=&
\frac{K_{lin}(i(i+1)+2-2i(\mu-\lambda)/\lambda+2\mu\cdot i/\lambda )}{2(\mu-\lambda)}\\
&=&\frac{K_{lin}(i(i+1)+2+2i}{2(\mu-\lambda)}=\frac{K_{lin}(i(i+1)+2(i+1)}{2(\mu-\lambda)}=\frac{K_{lin}(i+2)(i+1)}{2(\mu-\lambda)}.
\end{eqnarray*}
 By induction, the formula for $H_{\mu}(i)$ is correct.

\subsection{Quadratic holding costs
}\label{sect:HfVerSqd}

For quadratic holding cost function $h(i)=K_{sqd}\cdot i^2$, we propose and verify the solution 
\begin{equation}\label{hmusqd}
    H_{\mu}(i)=\frac{K_{sqd}\cdot i(i+1)( \mu+5\lambda+2i(\mu-\lambda) )}{6(\mu-\lambda)^2}
\end{equation}
 of Equation~(\ref{Hiter}). Recall $g=c_\mu+ K_{sqd}\rho_{\mu}(1+\rho_{\mu})/(1-\rho_{\mu})^2$.
We see $H_{\mu}(0)=0$ and
\begin{eqnarray*}
    H_{\mu}(1)&= \frac{K_{sqd}\cdot 2( \mu+5\lambda+2(\mu-\lambda) )}{6(\mu-\lambda)^2}=\frac{K_{sqd}\cdot3(\mu+\lambda)}{3(\mu-\lambda)^2}\\
    &= \frac{K_{sqd}(1+\rho_{\mu})/\mu}{(1-\rho_{\mu})^2}=\frac{K_{sqd}\rho_{\mu}(1+\rho_{\mu})}{\lambda(1-\rho_{\mu})^2} = \frac{g-c_\mu}{\lambda}.
\end{eqnarray*}
\allowdisplaybreaks
Then we use induction and derive/verify
\fontsize{10}{12}
\begin{eqnarray*}
H_{\mu}(i+1)&=&H_{\mu}(i)+\frac{g-c_\mu-h(i)+\mu(H_{\mu}(i)-H_{\mu}(i-1))}{\lambda}\\
&=&\frac{K_{sqd}\cdot i(i+1)( \mu+5\lambda+2i(\mu-\lambda) )}{6(\mu-\lambda)^2}+\frac{6K_{sqd}(\mu+\lambda)}{6(\mu-\lambda)^2}-\frac{K_{sqd}\cdot i^2}{\lambda}\\
&&\quad+\frac{K_{sqd}\mu(i(i+1)( \mu+5\lambda+2i(\mu-\lambda) )- i(i-1)( \mu+5\lambda+2(i-1)(\mu-\lambda) ) )/\lambda}{6(\mu-\lambda)^2} \\
&=& \frac{K_{sqd}}{6(\mu-\lambda)^2} \Bigg( i(i+1)( \mu+5\lambda+2i(\mu-\lambda) )+6(\mu+\lambda)-6i^2(\mu-\lambda)^2/\lambda\\
&&\hspace{2.5cm}+\Big(2i( \mu+5\lambda)+2(\mu-\lambda)((i+1)i^2-(i-1)^2i) \Big)(1+(\mu-\lambda)/\lambda)\Bigg)\\
&=&  \frac{K_{sqd}}{6(\mu-\lambda)^2} \Bigg((i+1)(i+2)(\mu+5\lambda)-2(\mu+5\lambda)+2(i+1)i^2(\mu-\lambda)  \\
&& \hspace{2.5cm} +6(\mu+\lambda)-6i^2(\mu-\lambda)^2/\lambda+2i(\mu-\lambda)(\mu+5\lambda)/\lambda\\
&&\hspace{6cm} +2(\mu-\lambda)(3i^2-i)(1+(\mu-\lambda)/\lambda) \Bigg)\\
&=&  \frac{K_{sqd}}{6(\mu-\lambda)^2} \Bigg((i+1)(i+2)(\mu+5\lambda)+4\mu-4\lambda+2(i+1)i^2(\mu-\lambda)  \\
&& \hspace{2.5cm} 
+6i^2(\mu-\lambda)-2i(\mu-\lambda)-2i(\mu-\lambda)^2/\lambda+2i(\mu-\lambda)(\mu+5\lambda)/\lambda
\Bigg)\\
&=&  \frac{K_{sqd}}{6(\mu-\lambda)^2} \Bigg((i+1)(i+2)(\mu+5\lambda)+2(i+1)i^2(\mu-\lambda)+4(\mu-\lambda)  \\
&& \hspace{2.5cm} 
+6i^2(\mu-\lambda)-2i(\mu-\lambda)+12i(\mu-\lambda)
\Bigg)\\
&=&  \frac{K_{sqd}}{6(\mu-\lambda)^2} \Bigg((i+1)(i+2)(\mu+5\lambda)+2(\mu-\lambda)\Big((i+1)i^2+2+3i^2+5i \Big) 
\Bigg)\\
&=&  \frac{K_{sqd}}{6(\mu-\lambda)^2} \Bigg((i+1)(i+2)(\mu+5\lambda)+2(\mu-\lambda)(i+2)(i+1)^2 
\Bigg)\\
&=&  \frac{K_{sqd}}{6(\mu-\lambda)^2} \Bigg((i+1)(i+2)(\mu+5\lambda+2(i+1)(\mu-\lambda))
\Bigg).
\end{eqnarray*}
\fontsize{11}{13.6}
 By induction, the formula for $H_{\mu}(i)$ for the quadratic holding costs is also correct.

 \section{The expected saved costs plotted for various parameters}\label{sect:ExtraPlots}

 \subsection{The discounted model}\label{sect:ExtraPlots1}

 For fixed rate $\mu=\mu_1$ and $\alpha=0.005$, Figures~\ref{fig:plotA1} and \ref{fig:plotA2} contain plots of the expected saved costs for the remaining combinations of parameters from Table~\ref{table1} and Table~\ref{table2}.

\begin{figure}[H]
\centering
\begin{subfigure}{.5\textwidth}
  \centering
  \includegraphics[width=1\linewidth]{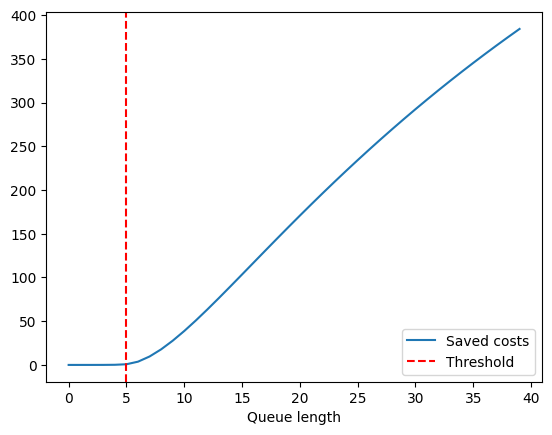}
\end{subfigure}%
\begin{subfigure}{.5\textwidth}
  \centering
  \includegraphics[width=1\linewidth]{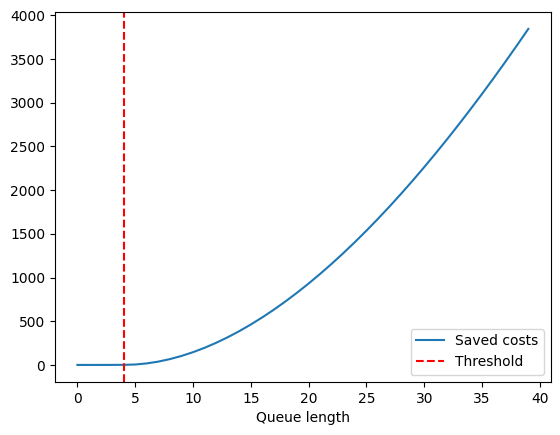}
\end{subfigure}
\caption{ Plots of (approximated) expected saved costs ($\tilde{s}_{\alpha,\epsilon}(i)$) for $\alpha=0.005,\beta=0.1, \lambda =0.1,\mu_1=0.35,\mu_2=0.45,c_{\mu_2}=10$, and $\epsilon=0.001$. The two cases correspond to linear holding cost function $h(i)=5\cdot i$, and quadratic holding cost function $h(i)=i^2$, respectively.}
\label{fig:plotA1}
\end{figure}

\begin{figure}[H]
\centering
\begin{subfigure}{.5\textwidth}
  \centering
  \includegraphics[width=1\linewidth]{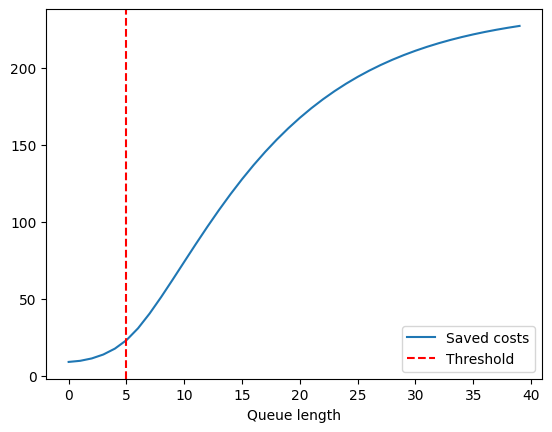}
\end{subfigure}%
\begin{subfigure}{.5\textwidth}
  \centering
  \includegraphics[width=1\linewidth]{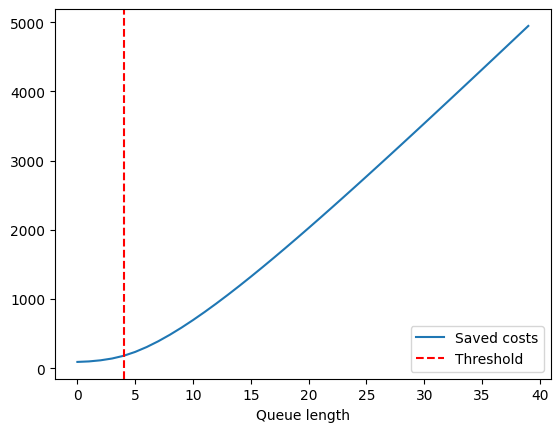}
\end{subfigure}
\caption{ Plots of (approximated) expected saved costs ($\tilde{s}_{\alpha,\epsilon}(i)$) for $\alpha=0.005,\beta=0.02, \lambda =0.31,\mu_1=0.33,\mu_2=0.34$, and $\epsilon=0.001$. The two cases correspond to linear holding cost function $h(i)=5\cdot i$ with $c_{\mu_2}=4$, and quadratic holding cost function $h(i)=i^2$ with $c_{\mu_2}=10$, respectively.}
\label{fig:plotA2}
\end{figure}

 Likewise, Figures~\ref{fig:plotA3} and \ref{fig:plotA4} contain plots for the case $\mu=\mu_2$ with the remaining combinations of parameters from Table~\ref{table3} and Table~\ref{table4}.

\begin{figure}[H]
\centering
\begin{subfigure}{.5\textwidth}
  \centering
  \includegraphics[width=1\linewidth]{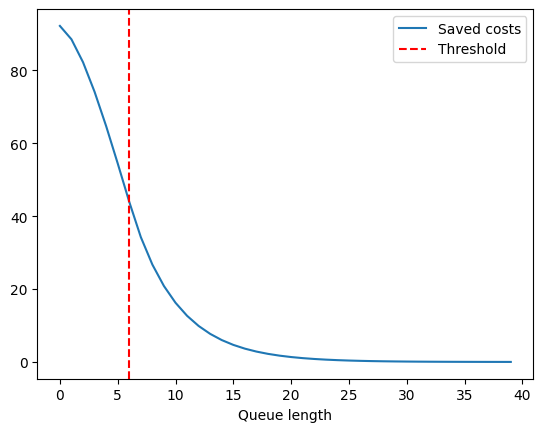}
\end{subfigure}%
\begin{subfigure}{.5\textwidth}
  \centering
  \includegraphics[width=1\linewidth]{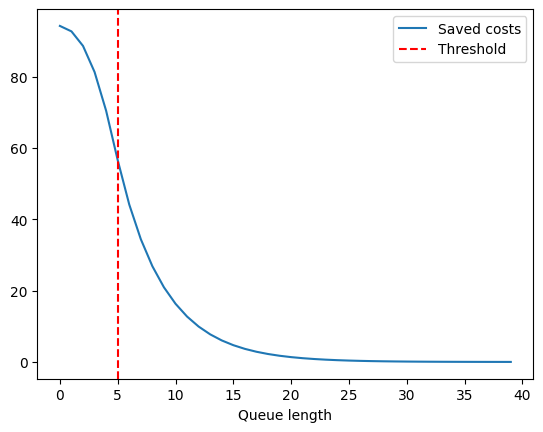}
\end{subfigure}
\caption{  Plots of (approximated) expected saved costs ($\tilde{s}_{\alpha,\epsilon}(i)$) for $\alpha=0.005,\beta=0.1, \lambda =0.1,\mu_1=0.35,\mu_2=0.45,c_{\mu_2}=10$, and $\epsilon=0.001$. The two cases correspond to linear holding cost function $h(i)=5\cdot i$, and quadratic holding cost function $h(i)=i^2$, respectively.}
\label{fig:plotA3}
\end{figure}

\begin{figure}[H]
\centering
\begin{subfigure}{.5\textwidth}
  \centering
  \includegraphics[width=1\linewidth]{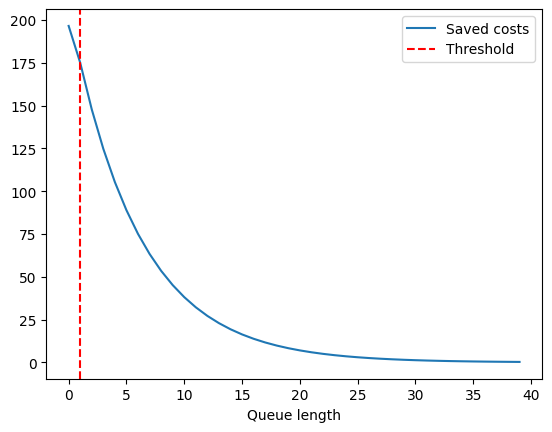}
\end{subfigure}%
\begin{subfigure}{.5\textwidth}
  \centering
  \includegraphics[width=1\linewidth]{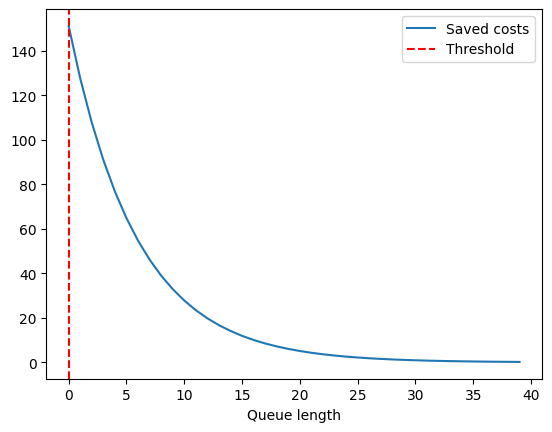}
\end{subfigure}
\caption{  Plots of (approximated) expected saved costs ($\tilde{s}_{\alpha,\epsilon}(i)$) for $\alpha=0.005,\beta=0.01, \lambda =0.33,\mu_1=0.3,\mu_2=0.36,c_{\mu_2}=10$, and $\epsilon=0.001$. The two cases correspond to linear holding cost function $h(i)=5\cdot i$, and quadratic holding cost function $h(i)=i^2$, respectively.}
\label{fig:plotA4}
\end{figure}

 \subsection{The non-discounted model}\label{sect:ExtraPlots2}

Figures~\ref{fig:plotA5} and \ref{fig:plotA6} portray the total expected saved costs for the case $\mu=\mu_1$ with the remaining combinations of parameters from Table~\ref{table5} and Table~\ref{table6}.

 \begin{figure}[H]
\centering
\begin{subfigure}{.5\textwidth}
  \centering
  \includegraphics[width=1\linewidth]{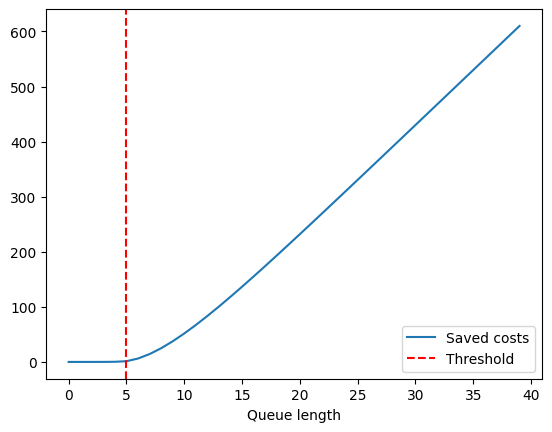}
\end{subfigure}%
\begin{subfigure}{.5\textwidth}
  \centering
  \includegraphics[width=1\linewidth]{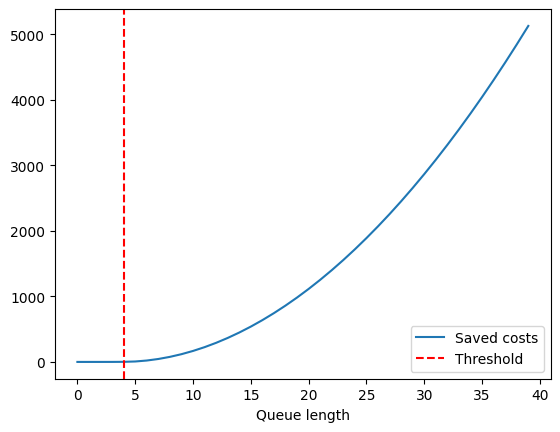}
\end{subfigure}
\caption{ Plots of (approximated) expected saved costs ($\tilde{s}_\epsilon(i)$) for $\beta=0.1, \lambda =0.1,\mu_1=0.35,\mu_2=0.45,c_{\mu_2}=10,\mu=\mu_1$ and $\epsilon=0.001$. The two cases correspond to linear holding cost function $h(i)=5\cdot i$, and quadratic holding cost function $h(i)=i^2$, respectively.}
\label{fig:plotA5}
\end{figure}

\begin{figure}[H]
\centering
\begin{subfigure}{.5\textwidth}
  \centering
  \includegraphics[width=1\linewidth]{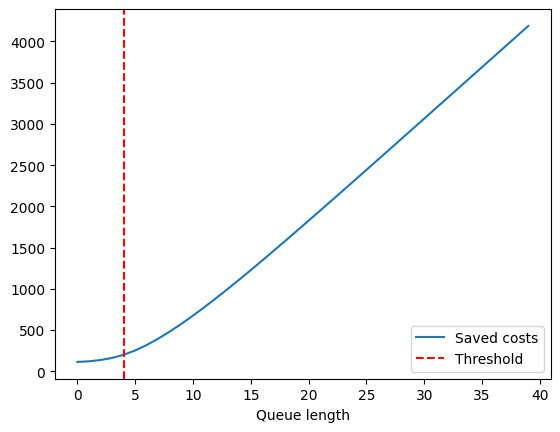}
\end{subfigure}%
\begin{subfigure}{.5\textwidth}
  \centering
  \includegraphics[width=1\linewidth]{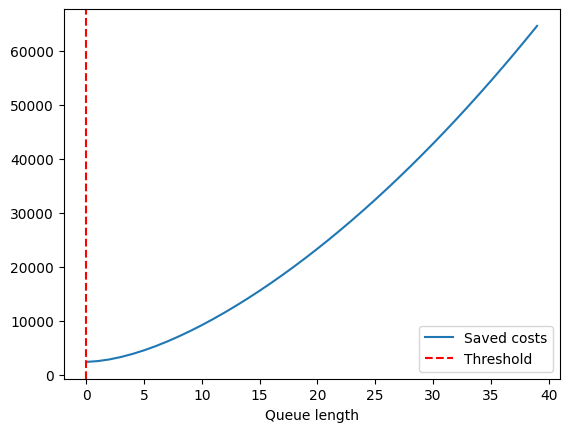}
\end{subfigure}
\caption{ Plots of (approximated) expected saved costs ($\tilde{s}_\epsilon(i)$) for $\beta=0.02, \lambda =0.31,\mu_1=0.33,\mu_2=0.34,c_{\mu_2}=10,\mu=\mu_1$ and $\epsilon=0.001$. The two cases correspond to linear holding cost function $h(i)=5\cdot i$, and quadratic holding cost function $h(i)=i^2$, respectively.}
\label{fig:plotA6}
\end{figure}

\begin{figure}[H]
\centering
\begin{subfigure}{.5\textwidth}
  \centering
  \includegraphics[width=1\linewidth]{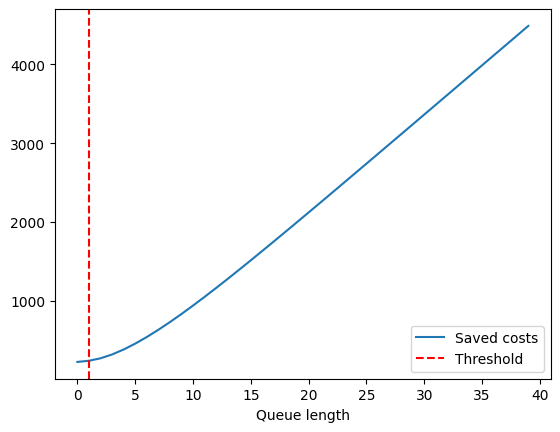}
\end{subfigure}%
\begin{subfigure}{.5\textwidth}
  \centering
  \includegraphics[width=1\linewidth]{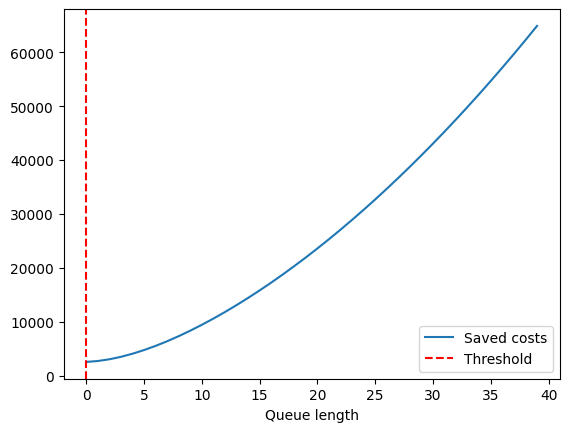}
\end{subfigure}
\caption{ Plots of (approximated) expected saved costs ($\tilde{s}_\epsilon(i)$) for $\beta=0.02, \lambda =0.31,\mu_1=0.33,\mu_2=0.34,c_{\mu_2}=4,\mu=\mu_1$ and $\epsilon=0.001$. The two cases correspond to linear holding cost function $h(i)=5\cdot i$, and quadratic holding cost function $h(i)=i^2$, respectively.}
\label{fig:plotA99}
\end{figure}

 Finally, for fixed rate $\mu=\mu_2$ 
 Figures~\ref{fig:plotA7} and \ref{fig:plotA8}
 contain plots of the total 
expected saved costs for the remaining combinations of parameters from Table~\ref{table7} and Table~\ref{table8}.

 \begin{figure}[H]
\centering
\begin{subfigure}{.5\textwidth}
  \centering
  \includegraphics[width=1\linewidth]{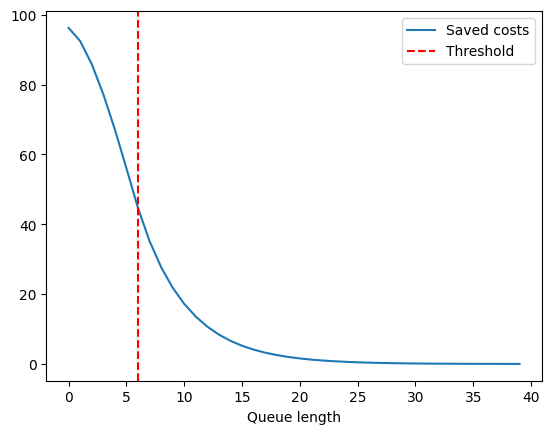}
\end{subfigure}%
\begin{subfigure}{.5\textwidth}
  \centering
  \includegraphics[width=1\linewidth]{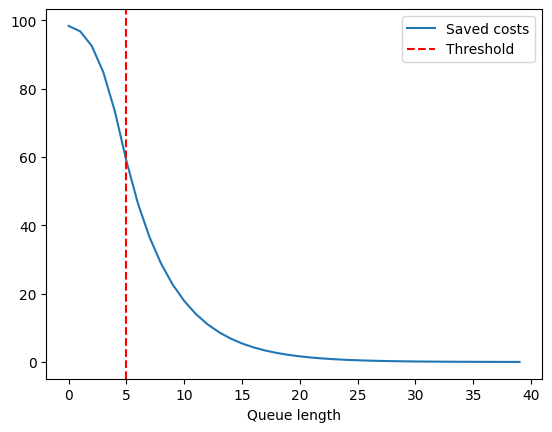}
\end{subfigure}
\caption{ Plots of (approximated) expected saved costs ($\tilde{s}_\epsilon(i)$) for $\beta=0.1, \lambda =0.1,\mu_1=0.35,\mu_2=0.45,c_{\mu_2}=10,\mu=\mu_2$ and $\epsilon=0.001$. The two cases correspond to linear holding cost function $h(i)=5\cdot i$, and quadratic holding cost function $h(i)=i^2$, respectively.}
\label{fig:plotA7}
\end{figure}

\begin{figure}[H]
\centering
\begin{subfigure}{.5\textwidth}
  \centering
  \includegraphics[width=1\linewidth]{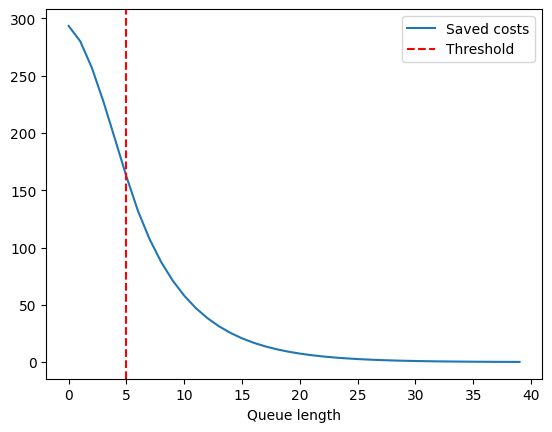}
\end{subfigure}%
\begin{subfigure}{.5\textwidth}
  \centering
  \includegraphics[width=1\linewidth]{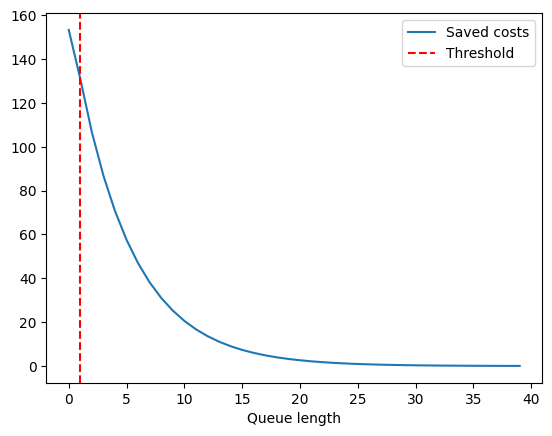}
\end{subfigure}
\caption{ Plots of (approximated) expected saved costs ($\tilde{s}_\epsilon(i)$) for $\beta=0.02, \lambda =0.31,\mu_1=0.33,\mu_2=0.34,c_{\mu_2}=10,\mu=\mu_2$ and $\epsilon=0.001$. The two cases correspond to linear holding cost function $h(i)=5\cdot i$, and quadratic holding cost function $h(i)=i^2$, respectively.}
\label{fig:plotA8}
\end{figure}

\begin{figure}[H]
\centering
\begin{subfigure}{.5\textwidth}
  \centering
  \includegraphics[width=1\linewidth]{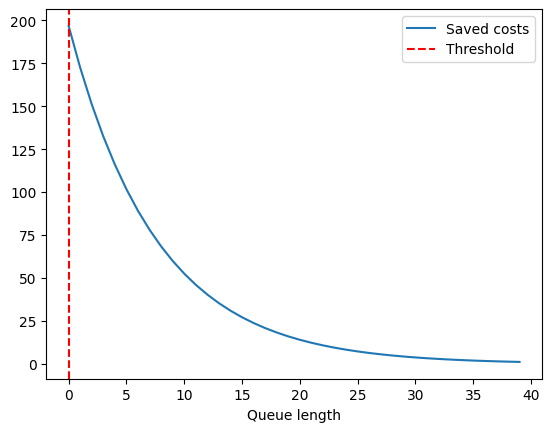}
\end{subfigure}%
\begin{subfigure}{.5\textwidth}
  \centering
  \includegraphics[width=1\linewidth]{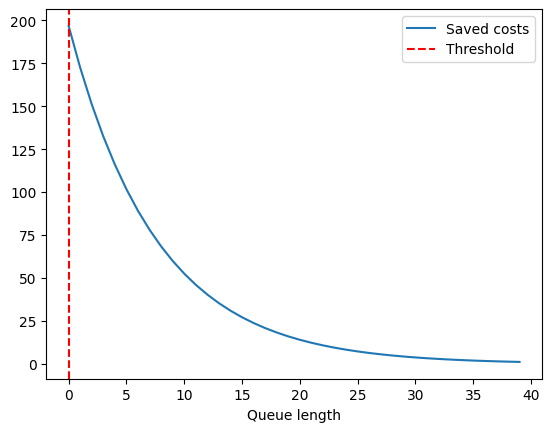}
\end{subfigure}
\caption{ Plots of (approximated) expected saved costs ($\tilde{s}_\epsilon(i)$) for $\beta=0.01, \lambda =0.33,\mu_1=0.3,\mu_2=0.36,c_{\mu_2}=10,\mu=\mu_2$ and $\epsilon=0.001$. The two cases correspond to linear holding cost function $h(i)=5\cdot i$, and quadratic holding cost function $h(i)=i^2$, respectively.}
\label{fig:plotA9}
\end{figure}

\end{document}